%% file: main.tex
\documentclass[12pt,preprint]{amsart}

\usepackage{amsmath, latexsym, amsfonts, amssymb, amsthm}
\usepackage{graphicx, color,hyperref,dsfont,epsfig,caption,wrapfig,subfig}
\usepackage{pstricks,yfonts,mathalfa}
\usepackage{mathtools}
\usepackage{parskip}
\usepackage{movie15}
\hypersetup{
	linktoc=page,
	linkcolor=red,          
	citecolor=blue,        
	filecolor=blue,      
	urlcolor=cyan,
	colorlinks=true           
}

\NewDocumentEnvironment{eqs}{+b}
    {\begin{equation}\begin{split}#1\end{split}\end{equation}}
    {}
\numberwithin{equation}{section}

\newtheorem{theorem}{Theorem}[section]
\newtheorem{lemma}[theorem]{Lemma}
\newtheorem{proposition}[theorem]{Proposition}

\newtheorem{remark}[theorem]{Remark}
\newtheorem{defi}[theorem]{Definition}

\newcounter{thmc}

\usepackage{simpler-wick}

\theoremstyle{definition}

\let\oldtocsection=\tocsection
\let\oldtocsubsection=\tocsubsection
\let\oldtocsubsubsection=\tocsubsubsection
\renewcommand{\tocsection}[2]{\hspace{0em}\oldtocsection{#1}{#2}}
\renewcommand{\tocsubsection}[2]{\hspace{1em}\oldtocsubsection{#1}{#2}}
\renewcommand{\tocsubsubsection}[2]{\hspace{2em}\oldtocsubsubsection{#1}{#2}}

\renewcommand{\tilde}{\widetilde}          
\DeclareMathSymbol{\leqslant}{\mathalpha}{AMSa}{"36} 
\DeclareMathSymbol{\geqslant}{\mathalpha}{AMSa}{"3E} 
\DeclareMathSymbol{\eset}{\mathalpha}{AMSb}{"3F}     
\renewcommand{\leq}{\;\leqslant\;}                   
\renewcommand{\geq}{\;\geqslant\;}                   

\newcommand{\C}{\mathbb{C}}
\newcommand{\R}{\mathbb{R}}

\newcommand{\D}{\mathbb{D}} 
\newcommand{\Heps}{\mathbb{H}_{\delta,\eps}}
\newcommand{\Reps}{\mathbb{R}_{\eps}} 
\renewcommand{\H}{\mathbb{H}}

\newcommand{\E}{\mathds{E}}
\newcommand{\X}{\bm{\mathrm X}}

\newcommand{\V}{\bm{\mathrm V}} 

\newcommand{\Desc}{\bm{\mathrm D}}

\newcommand{\ps}[1]{\langle #1 \rangle}
\newcommand{\mc}[1]{\mathcal{#1}}

\def\sl{\mathfrak{sl}}
\makeatletter
\newcommand{\ostar}{\mathbin{\mathpalette\make@circled\star}}
\newcommand{\make@circled}[2]{%
  \ooalign{$\m@th#1\smallbigcirc{#1}$\cr\hidewidth$\m@th#1#2$\hidewidth\cr}%
}
\newcommand{\smallbigcirc}[1]{%
  \vcenter{\hbox{\scalebox{0.77778}{$\m@th#1\bigcirc$}}}%
}
\makeatother

\def\Is{I_{sing} }
\def\Ir{I_{reg} }
\def\It{I_{tot} }

\newcommand{\qt}[1]{\quad\text{#1}\quad}
\def\weyl{\bm{\mathrm{\rho}}}

\def\V{\bm{\mathrm V}}

\def\Wb{\bm{\mathrm W}}

\def\SET{\bm{\mathrm T}}
\def\Wc{\bm{\mathcal W}}
\def\Lc{\bm{\mathcal L}}
\def\X{\bm{\mathrm  X}}

\def\L{\bm{\mathrm L}}

\def\eps{\varepsilon}

\def\g{\mathfrak{g}}

\def\wc{\mathcal C_-}

\def\eps{\varepsilon}

\def\bi{\begin{itemize}}
	\def\ei{\end{itemize}}

\def\bnum{\begin{enumerate}}
	\def\enum{\end{enumerate}}

\def\<#1{\langle #1 \rangle}

\newcommand{\norm}[1]{\left\lvert#1\right\rvert}
\newcommand{\expect}[1]{\mathbb{E}\left[#1\right]}
\usepackage{bm}
\usepackage{bbm}

\title[Higher-spin symmetry in the $\sl_3$ boundary Toda CFT I]{Higher-spin symmetry in the $\sl_3$ boundary Toda conformal field theory I: Ward identities}

\author{Baptiste Cercl\'e}
\email{baptiste.cercle@epfl.ch}
\address{EPFL SB MATH RGM, MA B2 397, Station 8, CH-1015 Lausanne, Switzerland.}
\author{Nathan Huguenin}
\email{nathan.huguenin@univ-amu.fr}
\address{Aix-Marseille Université, CNRS, Institut de Mathématiques de Marseille (I2M) – UMR 7373, Site de Saint Charles, 3 place Victor Hugo, Case 19, 13331 Marseille cédex 3, France.}

\begin{document}

	\maketitle
	\begin{abstract}
		This article is the first of a two-part series dedicated to studying the symmetries enjoyed by the probabilistic construction of the $\sl_3$ boundary Toda Conformal Field Theory. Namely in the present document we show that this model enjoys higher-spin symmetry in the form of Ward identities, both local and global. 

        To do so we consider the $\sl_3$ Toda theory on the upper-half plane and rigorously define the descendant fields associated to the Vertex Operators. We then show that we can express local as well as global Ward identities based on them, for both the stress-energy tensor and the higher-spin current that encodes this enhanced level of symmetry. This answers a question raised in the physics literature as to whether Toda theory still enjoys higher-spin symmetry in the boundary case. 

        The second part of this series will be dedicated to computing the singular vectors of the theory and showing that they give rise to higher equations of motion as well as, under additional assumptions, BPZ-type differential equations for the correlation functions.
	\end{abstract}

    \renewcommand{\baselinestretch}{0.75}\normalsize
    \tableofcontents
    \renewcommand{\baselinestretch}{1.0}\normalsize

    \input{intro}
	
	\input{background.tex}

    \input{proof_ward.tex}

	\bibliography{main}
	\bibliographystyle{plain}
	
\end{document}

%% file: intro.tex
\section{Introduction}

\subsection{Higher-spin symmetry for boundary Toda CFT} 

\subsubsection{Two-dimensional conformal field theory} Two-dimensional conformal field theories (CFT hereafter) are celebrated instances of physics models in which the symmetries are constraining enough to solve the theory. They appear as fundamental objects in many topics such as string theory, 2d quantum gravity and 2d statistical physics at criticality when a second-order (at least) phase transition is observed. The conformal bootstrap method developed by Belavin-Polyakov-Zamolodchikov~\cite{BPZ} (BPZ in the sequel) arises in this setting as a general procedure that should allow one to recursively compute the correlation functions of a CFT based on the knowledge of basic correlation functions called the \emph{structure constants}, together with the so-called \textit{spectrum} of the theory as well as some universal special functions named \emph{conformal blocks}. To this end, the general philosophy for computing these correlation functions is to exploit the symmetries of the model, which can be achieved \textit{e.g.} by studying the representation theory of the Virasoro algebra (the algebra of symmetry of 2d CFTs), and which in turn translates as actual constraints put on the correlation functions. Such constraints can manifest themselves through conformal \emph{Ward identities}, and under additional assumptions made on the representation considered they give rise to \emph{BPZ differential equations}. This approach has been successfully worked out (at the physics level of rigor) by BPZ~\cite{BPZ} to address the case of \textit{minimal models}.

In this setting, Liouville theory is a canonical model for which the conformal bootstrap procedure can be implemented. Initially introduced by Polyakov~\cite{Pol81} as a model for two-dimensional random geometry, it has since become a fundamental object and appears in various topics, ranging from gauge theories via the AGT correspondence~\cite{AGT} to Wess-Zumino-Witten models. 
The first step in the conformal bootstrap procedure, that is the computation of the structure constants of Liouville theory, has been conducted in the physics literature in~\cite{DO94, ZZ96} for the closed case, and later in~\cite{FZZ, PT02, Hos} for the boundary case. This led to the implementation of the conformal bootstrap procedure: for more details on the physical formulation of the conformal bootstrap in Liouville theory we refer for instance to~\cite{Teschner_revisited} and to the review~\cite{Nakayama}.

\subsubsection{Toda conformal field theories} 
Beyond conformal invariance, it is well-known in statistical physics that many models admit an enhanced level of symmetry: this is the case for example for the three-states Potts model at criticality. In order to generalize the BPZ method to these theories, Zamolodchikov introduced in~\cite{Za85} the notion of \emph{$W$-algebra}, that are vertex operator algebras that encode the \emph{higher-spin} symmetry of the model, and that strictly contain the Virasoro algebra as a sub-algebra. Toda conformal field theories are natural generalizations of Liouville theory that exhibit this higher-spin or $W$-symmetry: these CFTs depend on a simple and complex Lie algebra $\mathfrak{g}$ and their algebras of symmetry are given by the simple $W$-algebras $\mc W^k(\mathfrak{g})$. Liouville theory is then recovered within this framework as the Toda theory associated with $\mathfrak{sl}_2$. 

It is worth noting at this point that unlike the Virasoro algebra, $W$-algebras are \textit{not} Lie algebras and the formalism of Vertex Operator Algebra is thus necessary to make sense of this notion. As such, the study of the symmetries enjoyed by Toda CFTs is more involved than its counterpart in Liouville theory and relies on additional tools related, among other things, to the representation theory of $W$-algebras. In the case of the $\g=\sl_3$ boundary Toda theory that we consider here, in order to make explicit the symmetries enjoyed by the model, one needs to work out Ward identities associated not only with conformal invariance but also with higher-spin symmetry. Actually, it was even far from clear in the physics literature that the model that we consider here has the desired symmetries (in the sense of the Ward identities): one achievement of the present document is to address this issue by showing that it is actually the case. Moreover, a necessary step in the comprehension of the model is to find the explicit form of some singular vectors of the theory, which were also unknown up to now. But let us first describe Toda CFTs in more details.

In certain cases these models admit a path integral formulation, that reads
\begin{equation*}
    \ps{F[\Phi]} \coloneqq \frac{1}{\mathcal{Z}} \int_{\mathcal{F}} F[\phi] e^{-S_T(\phi)} D\phi.
\end{equation*}
where $\mathcal{F}$ is a space of maps defined on $\Sigma$ and taking values in $\mathfrak{a}\simeq \R^r$, that is the real part of the Cartan sub-algebra of $\mathfrak{g}$, and where the action $S_T$ is given by
\begin{equation*}
    \begin{split}
    S_T(\phi) :=& \frac{1}{4\pi} \int_\Sigma \left( |d_g\phi|^2+R_g\ps{Q,\phi} + 4\pi\sum_{i=1}^r\mu_{B,i} e^{\ps{\gamma e_i,\phi}}\right)dv_g \\
    &+ \frac{1}{2\pi} \int_{\partial\Sigma} \left( k_g\ps{Q,\phi} + 2\pi \sum_{i=1}^r\mu_i e^{\ps{\frac{\gamma}{2}e_i,\phi}}\right)d\lambda_g.
    \end{split}
\end{equation*}
We stress that we consider here the general case where the boundary is non-empty, so that the term on the second line is indeed present in the definition of the model. This action functional is based on several quantities, arising from different perspectives:
\begin{itemize}
    \item from the Lie algebra perspective, the $e_i$'s are the simple roots of the Lie algebra $\mathfrak{g}$, and the scalar product $\ps{\cdot,\cdot}$ on $\mathfrak{a}$ is inherited from the Killing form;
    \item the geometric quantities are $g$ a Riemannian metric on $\Sigma$, with $R_g$ and $k_g$ being respectively the Ricci scalar and geodesic curvature associated with $g$;
    \item the \emph{coupling constant} $\gamma$ is a real number in $(0,\sqrt{2})$~\footnote{Due to the normalization of the simple roots $|e_i|^2=2$, the range of values for $\gamma$ differs from the one in Liouville theory by a factor $\sqrt{2}$. The so-called $L^2$-phase in the $\mathfrak{sl}_3$ Toda theory is thus the interval $(0,1)$.}, and the \emph{background charge} is given by $Q := \left(\gamma+\frac{2}{\gamma}\right)\boldsymbol{\rho}$, where $\boldsymbol{\rho}$ is the Weyl vector;
    \item eventually the action depends on \emph{cosmological constants} $\mu_{B,i} \ge 0$ for $i=1,2$ and $\mu_i$ that are piecewise constant, complex-valued functions over $\partial\Sigma$. 
\end{itemize}

Among all observables $F[\Phi]$, some are of particular interest and are called \emph{Vertex Operators}. They depend on a point $z$ in $\Sigma \cup \partial\Sigma$ and a weight $\alpha\in \R^r$ and are formally defined by
\begin{equation*}
    V_\alpha(z)[\Phi] \coloneqq e^{\ps{\alpha,\Phi(z)}}.
\end{equation*}
The correlation functions of Vertex Operators are then defined by considering distinct bulk insertions $(z_1,\cdots,z_N) \in \H^N$ and distinct boundary insertions $(s_1,\cdots,s_M)\in\R^M$ (here $\R$ is viewed as the boundary of $\H$) with respective weights $(\alpha_1,\cdots,\alpha_N)$ and $(\beta_1,\cdots,\beta_M)$. They formally take the form
\begin{equation*}
    \ps{\prod_{k=1}^NV_{\alpha_k}(z_k)\prod_{l=1}^MV_{\beta_l}(s_l)}\coloneqq \frac{1}{\mc Z}\int_{\mc F}\prod_{k=1}^Ne^{\ps{\alpha_k,\Phi(z_k)}}\prod_{l=1}^Me^{\ps{\beta_l,\Phi(s_l)}}.
\end{equation*}
Computing such correlation functions is one of the main goals in the study of Toda CFTs. The first step in this perspective is the derivation of the structure constants of the theory, in which case a formula for the structure constants on the sphere has been proposed by Fateev-Litvinov~\cite{FaLi1} when $\mathfrak{g}=\mathfrak{sl}_n$. For the theory with a boundary, the structure constants are not known except for the bulk one point function that has been computed in~\cite{FaRi} for a certain type of boundary conditions, see also~\cite{Fre11} where other type of boundary conditions are considered and where the link with $W_n$ minimal models is investigated. And beyond the structure constants, it is not clear in which form the recursive procedure at the heart of the conformal bootstrap method would hold.

In order to derive such expressions, one relies on the symmetries of the theory which are encoded by the $W_3$ algebra. In the model, this $W_3$ algebra manifests itself via the existence of two holomorphic currents, the stress-energy tensor $\SET$ associated to conformal invariance and the higher-spin current $\Wb$ that encodes the higher level of symmetry.  At the (vertex operator) algebraic level, these currents admit the formal Laurent series expansion
\begin{equation*}
\SET(z)=\sum_{n\in\mathbb Z}z^{-n-2}\L_{n}\qt{and}\Wb(z) =\sum_{n\in\mathbb{Z}}z^{-n-3}\Wb_n
\end{equation*}
where the modes $(\L_n,\Wb_m)_{n,m\in\mathbb Z}$ satisfy the commutation relations of the $W_3$ algebra. One key property of these currents is their \emph{Operator Product Expansions} with Vertex Operators, which is assumed to take the form:
\begin{equation}\label{eq:WVOPEinformal}
    \bm {\mathrm W}(z_0)V_\alpha(z)=\frac{w(\alpha)V_\alpha(z)}{(z_0-z)^3}
    +\frac{\bm {\mathrm W}_{-1}V_\alpha(z)}{(z_0-z)^2}+\frac{\bm {\mathrm W}_{-2}V_\alpha(z)}{z_0-z}+\text{reg.}
\end{equation}
where the $\bm{\mathrm{W}_{-i}}V_\alpha(z)$ are the \emph{descendant fields} associated to $V_\alpha$ while $w(\alpha)\in\C$ is called the \emph{quantum number} associated to $\Wb$. When inserted within correlation functions this Operator Product Expansion becomes a \textit{local Ward identity}.

\subsubsection{Probabilistic approaches to conformal field theory} The physics picture described above has been for many years lacking rigorous mathematical foundation. However the probabilistic framework developed by David-Guillarmou-Kupiainen-Rhodes-Vargas~\cite{DKRV} for Liouville theory have led to major progress in the mathematical understanding of two-dimensional conformal field theory. Firstly, in the closed setting (that is, without boundary), the DOZZ formula for the sphere structure constant has been rigorously derived in \cite{KRV_DOZZ}, and the conformal boostrap has been implemented in \cite{GKRV}, while the Segal's axioms~\cite{Seg04} were proven to be true in \cite{GKRV_Segal}. In the boundary case, the structure constants of Liouville theory were obtained in the case where $\mu_B=0$ by Remy~\cite{remy1} and Remy-Zhu~\cite{remy2} using the BPZ equations. In the case where $\mu_B> 0$, the mating-of-trees machinery introduced in \cite{DMS14} has led to the derivation of BPZ equations and to the computation of all structure constants in \cite{ARS, ARSZ}. Recently, the first author has provided a framework allowing to obtain these BPZ equations avoiding the use of the mating-of-trees techniques \cite{Cer_HEM}, paving the way for a derivation of such differential equations for the case of Toda CFTs. 

Indeed, the probabilistic construction unveiled in~\cite{DKRV} admits a generalization to Toda theories on the Riemann sphere as defined in~\cite{Toda_construction}. In the case of the Lie algebra $\mathfrak{sl}_3$, which we will consider in this document, the symmetries of the probabilistic model thus defined have been unveiled in~\cite{Toda_OPEWV} where Ward identities were proven to hold, leading to the computation of a family of three-point structure constants in~\cite{Toda_correl1, Toda_correl2}. One goal of the present paper is to conduct the same analysis in the case where the surface under consideration admits a boundary, which is both conceptually and technically more demanding that the closed case. For this we rely on the general construction of Toda theories on surfaces with or without boundary performed in~\cite{CH_construction} and show that the model constructed there enjoys $W$-symmetry as well. More precisely, a first outcome of the present paper is to show that the \emph{boundary} local Ward identities are valid~\footnote{We also show that the bulk Ward identities are true as well, but the proof of this result is contained in the one for the boundary case and is much simpler.} for the theory on the complex upper half-plane (the boundary thus being the real line) and associated with the Lie algebra $\mathfrak{sl}_3$. We stress that it was far from clear in the physics literature that such Ward identities would hold in the boundary case.
\begin{theorem}\label{thm:ward_intro}
    Let $t\in \R$. For any $n\ge 2$, in the sense of weak derivatives, the conformal Ward identity holds:
    \begin{equation*}
        \begin{split}
            &\ps{\L_{-n}V_\beta(t)\prod_{k=1}^NV_{\alpha_k}(z_k)\prod_{l=1}^MV_{\beta_l}(s_l)}\\
            &=\left(\sum_{k=1}^{2N+M}\frac{-\partial_{z_k}}{(z_k-t)^{n-1}}+\frac{(n-1)\Delta_{\alpha_k}}{(z_k-t)^n}\right) \ps{V_\beta(t)\prod_{k=1}^NV_{\alpha_k}(z_k)\prod_{l=1}^MV_{\beta_l}(s_l)}.
        \end{split}
    \end{equation*}
    Moreover, for any $n\geq3$, the higher-spin Ward identity is valid:
    \begin{equation*} \label{eq:ward_Wn_intro}
        \begin{split}
            &\ps{\Wb_{-n}V_\beta(t)\prod_{k=1}^NV_{\alpha_k}(z_k)\prod_{l=1}^MV_{\beta_l}(s_l)}= \\
              &\left(\sum_{k=1}^{2N+M}\frac{-\Wc_{-2}^{(k)}}{(z_k-t)^{n-2}}+\frac{(n-2)\Wc_{-1}^{(k)}}{(z_k-t)^{n-1}}-\frac{(n-1)(n-2)w(\alpha_k)}{2(z_k-t)^n}\right) \ps{V_\beta(t)\V}
        \end{split}
    \end{equation*}
    where for $j=1,2$, $\Wc_{-j}^{(k)}\ps{V_\beta(t)\V}\coloneqq \ps{\Wb_{-j}V_{\alpha_k}(z_k)\prod_{l\neq k}V_{\alpha_l}(z_l)}$, and $\Wc_{-j}^{(k)} = \overline{\Wc}_{-j}^{(k-N)}$ for $k\in\{N+1,...,2N\}$.
\end{theorem}
To prove this result, we first provide a rigorous definition of the boundary descendants $\L_{-n}V_\beta$ and $\Wb_{-n}V_\beta$ in the same fashion as in~\cite{Cer_HEM}. To do so we rely on the explicit expression of the currents $\SET$ and $\Wb$ in terms of the Toda field. Defining these descendants then relies on an algorithmic method that involves rather heavy computations, but these complicated expressions are seen to simplify thanks to the special observables (that is the stress-energy tensor and the higher-spin current) that we consider.

As a corollary of the local Ward identities and the conformal covariance of the model, we also obtain global Ward identities, that provide linear relations between the descendants associated with a given correlation function:
\begin{theorem}\label{thm:ward_global_intro}
    For $0\leq n\leq 2$ and $0\leq m\leq 4$:
    \begin{equation*}
        \begin{split}
            &\left(\sum_{k=1}^{2N+M}z_k^n\Lc_{-1}^{(k)}+nz_k^{n-1}\Delta_{\alpha_k}\right)\ps{\prod_{k=1}^NV_{\alpha_k}(z_k)\prod_{l=1}^MV_{\beta_l}(s_l)}=0\\
            &\left(\sum_{k=1}^{2N+M}z_k^m\Wc_{-2}^{(k)}+mz_k^{m-1}\Wc_{-1}^{(k)}+\frac{m(m-1)}2 z_k^{m-2}w(\alpha_k)\right)\ps{\prod_{k=1}^NV_{\alpha_k}(z_k)\prod_{l=1}^MV_{\beta_l}(s_l)}=0.
        \end{split}
    \end{equation*}
\end{theorem}

\subsection{Singular vectors and towards integrability}
In the companion paper \cite{CH_sym2}, we show the existence of \emph{singular vectors} in the theory, that under certain conditions give rise to \emph{null vectors}. The two types of constraints thus obtained, coming either from the Ward identities or from singular vectors, provide a lot of information about the model. Indeed, by combining these two properties of the theory we are able to derive BPZ-type differential equations for some correlation functions that contain these degenerate fields, but more generally we obtain higher equations of motion previously unknown in the physics literature. 

\subsubsection{From Ward identities to differential equations}
From the local Ward identities as stated in Theorem~\ref{thm:ward_intro} we know that we can express the insertion of a field $\Wb_{-3}V_{\beta}$ within a correlation function in terms of the descendants at order $2$ associated with the other insertions. On the other hand if $\beta=\beta^{\ostar}$ is chosen at a particular value (called \emph{fully-degenerate weight}) we know from \cite{CH_sym2} that it is possible to define singular vectors out of the descendant fields associated to $V_{\beta^{\ostar}}$. Under additional assumptions (on the cosmological constants) these singular vectors are actually null vectors, meaning that the $\Wb_{-3}V_{\beta^{\ostar}}$ descendant is actually equal to a linear combination of Virasoro descendants. This entails using the global Ward identities from Theorem~\ref{thm:ward_global_intro} that we are able to express, under suitable conditions put on the other insertions such as taking additional semi-degenerate fields, both sides of the Ward identities associated to $\Wb_{-3}$ in terms of Virasoro descendants. From this we show in~\cite{CH_sym2} that a family of correlation functions are solutions of BPZ-type differential equations.

\subsubsection{Towards integrability for the boundary $\sl_3$ Toda CFT}
The existence of differential equations satisfied by the above correlation functions paves the way for a rigorous derivation of a family of structure constants for the boundary $\sl_3$ Toda CFT. To be more specific we expect to be able to provide explicit expressions for the correlation functions of the form $\ps{V_\alpha(i)V_{\beta^*}(0)}$ (corresponding to the bulk-boundary correlator) as well as correlation functions of the form $\ps{V_{\beta_1}(0)V_{\beta^*}(1)V_{\beta_2}(\infty)}$. But before being able to derive exact formulas for such quantities a first step is the derivation of the boundary reflection coefficients. Indeed and in the same fashion as in~\cite{Toda_correl2}, in order for the correlation functions involved to be well-defined one first needs to extend the range of validity for the weights under which the correlation functions make sense. Along the same way as in~\cite{Toda_correl2} this would involve the consideration of such reflection coefficients, and the proof that they describe the joint tail expansion of both bulk and boundary correlated GMC measures.

Finally let us mention that none of these formulas is actually known in the physics literature, apart from the one-point function disclosed in~\cite{FaRi} and the conjectures stated in~\cite{Fre11} on the basis of computations on the minimal models side. In this respect being able to obtain exact formulas for some structure constants for boundary Toda CFT would be a major achievement.

\textit{\textbf{Acknowledgements:}}	
The authors would like to thank Rémi Rhodes for his interest and support during the preparation of the present manuscript.

B.C. has been supported by Eccellenza grant 194648 of the Swiss National Science Foundation and is a member of NCCR SwissMAP. B.C. would like to thank Université d'Aix-Marseille for their hospitality (despite the bad weather) and N.H. is grateful to the \'Ecole Polytechnique Fédérale de Lausanne where part of this work has been undertaken.

%% file: background.tex
\section{Description of the method and some technical estimates}

    \subsection{The general framework}
    Before defining the correlation functions of Toda CFT on the upper-half plane we need to introduce the objects that we will work with in the sequel.
    
    \subsubsection{Gaussian Free Fields} \label{subsubsec: gff}
    The probabilistic definition of boundary Toda CFT as performed in \cite{CH_construction} is based on the consideration of a vectorial Gaussian Free Field $\X:\overline\H\to\R^2$. This GFF is a centered Gaussian random distribution, and as such is often defined through its covariance kernel, given for any $u,v\in\R^2$ and $x,y$ in $\overline\H$ by
	\begin{equation*}
		\begin{split}&\expect{\ps{u,\X(x)}\ps{v,\X(y)}}=\ps{u,v}G(x,y),\quad\text{with }\\
        &G(x,y)\coloneqq\ln\frac{1}{\norm{x-y}\norm{x-\bar y}}+2\ln\norm{x}_++2\ln\norm{y}_+
        \end{split}
	\end{equation*} 
	and where we have used the notation $\norm{x}_+\coloneqq\max(\norm{x},1)$. The Green function $G$ can be alternatively written as 
    \begin{equation*}
        G(x,y) = G_{\hat{\C}}(x,y)+G_{\hat{\C}}(x,\bar{y}),\qt{with}G_{\hat{\C}}(x,y)=\ln\frac{1}{\norm{x-y}}+\ln\norm{x}_++\ln\norm{y}_+.
    \end{equation*}
    Here $G_{\hat{\C}}$ is the Green function on the Riemann sphere: in particular the GFF $\X$ has \textit{Neumann} boundary conditions. The object thus defined is not a function but rather an element of the Sobolev space with negative index $H^{-1}(\H)$. We can however construct a regular function out of it by considering a smooth mollifier $\eta$ and defining for any $\rho>0$:
	\begin{equation}\label{eq:regularized}
		\X_\rho(x)\coloneqq\int_\H\X_\rho(y)\eta_\rho(x-y)dyd\bar y
	\end{equation}
	where we have set $\eta_\rho(\cdot)\coloneqq\frac1{\rho^2}\eta(\frac{\cdot}{\rho})$. 

    \subsubsection{Gaussian Multiplicative Chaos}
	The exponential of the above GFF $\X$ also enters in the definition of Toda correlation functions. In the probabilistic setting this is achieved thanks to the theory of \textit{Gaussian Multiplicative Chaos}, which  allows to make sense of the following random measures on either $\H$ or its boundary $\R$:
	\begin{equation*}
		M_{\gamma e_i}(d^2x)\coloneqq\lim\limits_{\rho\to0}\rho^{\gamma^2}e^{\ps{\gamma e_i,\X_\rho(x)}}dxd\bar x;\quad M^\partial_{\gamma e_i}(dx)\coloneqq\lim\limits_{\rho\to0}\rho^{\frac{\gamma^2}{2}}e^{\ps{\frac\gamma2 e_i,\X_\rho(x)}}dx
	\end{equation*}
	where these limits hold in probability, in the sense of weak convergence of measures~\cite{Ber,RV_GMC}. In the above the $(e_i)_{i=1,2}$ are elements of $\R^2$ such that $\norm{e_i}^2=2$, and as such we need to make the assumption that $\gamma<\sqrt 2$ in order for these random measures to be well-defined.

    \subsubsection{On the Lie algebra $\mathfrak{sl}_3$}
	The elements $(e_1,e_2)$ that we just considered actually define a special basis of $\R^2$. These naturally arise in the consideration of the $\mathfrak{sl}_3$ Lie algebra where they correspond to so-called \emph{simple roots}. They satisfy the property that 
    \begin{equation*}
		\left(\ps{e_i,e_j}\right)_{i,j}\coloneqq A=\begin{pmatrix}
			2 & -1\\
			-1 & 2
		\end{pmatrix}
	\end{equation*}
    with $A$ the Cartan matrix of $\mathfrak{sl}_3$. We will also consider the dual basis of $(e_1,e_2)$ which we denote by $(\omega_1,\omega_2)$. Explicit computations show that we have
	\begin{equation*}
		\omega_1=\frac{2e_1+e_2}3\quad\text{and}\quad\omega_2=\frac{e_1+2e_2}{3}\cdot
	\end{equation*}
	In the case of $\mathfrak{sl}_3$, the Weyl vector is given by
	\begin{equation*}
		\weyl= \omega_1+\omega_2=e_1+e_2,\qt{so that}\ps{\weyl,e_i}=1\qt{for}i=1,2.
    \end{equation*}
    In this setting the background charge $Q$ can be expressed in terms of the Weyl vector and the coupling constant $\gamma\in(0,\sqrt2)$ as
    \begin{equation*}
        Q=\left(\gamma+\frac2\gamma\right)\weyl.
    \end{equation*}
	We will also consider the fundamental weights in the first fundamental representation $\pi_1$ of $\mathfrak{sl}_3$ with the highest weight $\omega_1$:
	\begin{equation}\label{eq:definition_hi}
		h_1\coloneqq \frac{2e_1+e_2}{3},\quad h_2\coloneqq \frac{-e_1+e_2}{3}, \quad h_3\coloneqq  -\frac{e_1+2e_2}{3}\cdot
	\end{equation}
	Finally we introduce the following notations that will enter the definitions of the higher-spin current:
	\begin{equation*}
		\begin{split}
			&B(u,v)\coloneqq (h_2-h_1)(u)h_1(v)+(h_3-h_2)(u)h_3(v);\\
			&C(u,v,w)\coloneqq h_1(u)h_2(v)h_3(w)+h_1(v)h_2(w)h_3(u)+h_1(w)h_2(u)h_3(v).
		\end{split}
	\end{equation*}
    We also introduce here the following shortcut:
    \begin{equation*}
        C^\sigma(u,v,w) := C(u,v,w) + C(u,w,v)
    \end{equation*}
    
    \subsection{Definition of the (regularized) correlation functions and descendant fields} In this section we sketch the construction of the Toda CFT on the upper half-plane $\H$ as introduced in~\cite{CH_construction}. We also define the regularized descendant fields that we wish to define in the present document.

    \subsubsection{Toda correlation functions on $\H$} We first define the \emph{Toda field} on $\overline{\H}$ to be given by\footnote{This definition amounts to considering the Toda CFT on $\H$ with the metric $g(z)=d^2z/|z|_+^4$, as it is often done.} 
    \begin{equation*}
        \Phi \coloneqq \X - 2Q\ln \norm{\cdot}_+ +\bm{c},
    \end{equation*}
    where $\X$ is the GFF defined previously, and $\bm{c}$ is distributed according to the Lebesgue measure on $\R^2$. We denote by $\Phi_\rho$ the regularized Toda field in the sense of Equation~\eqref{eq:regularized}. The Vertex Operators are observables of the Toda field that are of fundamental importance. They depend on an insertion point $z\in\H$ or $s\in\R$ as well as a weight $\alpha$ or $\beta$ in $\R^2$, and are formally defined by considering the functionals
    \begin{equation*}
        V_\alpha(z)[\Phi]\coloneqq e^{\ps{\alpha,\Phi(z)}}\qt{for $z\in\H$, while for $t\in\R$}V_\beta(s)[\Phi]\coloneqq e^{\ps{\frac\beta2,\Phi(s)}}.
    \end{equation*}
    Based on these observables we can form the correlation functions of Vertex Operators, which depend on insertion points $(z_1,...,z_N,s_1,...,s_M) \in \H^N\times \R^M$, to which are associated weights \\
    $(\alpha_1,...,\alpha_N,\beta_1,...,\beta_M)\in (\R^2)^{N+M}$. The corresponding correlation function is then formally given by
    \begin{equation*}
        \ps{\prod_{k=1}^NV_{\alpha_k}(z_k)\prod_{l=1}^MV_{\beta_l}(s_l)}
    \end{equation*}
    where the average is made with respect to the law of the Toda field. We now recall how to provide a rigorous meaning to the latter.
    
    For the sake of simplicity we write $\bm z=(z_k)_{k=1,...,2N+M}$ for $(z_1,...,z_N,\bar{z}_1,...,\bar{z}_N,s_1,...,s_M)$ and $\bm\alpha=(\alpha_k)_{k=1,...,2N+M}$ for $(\alpha_1,...,\alpha_N,\alpha_1,...,\alpha_N,\beta_1,...,\beta_M)$. In order to define the correlations we also introduce the bulk cosmological constants $\mu_{B,i} > 0$ for $i=1,2$, as well as the boundary cosmological constants $\mu_{i,l}$, which we choose in $\C$ but with the additional assumption that $\Re(\mu_{i,l})\ge 0$, and set $\mu_i$ to be the piecewise constant measure on $\R$ given by
    $$
    \mu_i(dx) = \sum_{l=1}^M \mu_{i,l+1}\mathds{1}_{(s_l,s_{l+1})}(x) dx
    $$
    where by convention $s_0=-\infty$, $s_{M+1}=+\infty$ and $\mu_{M+1}=\mu_1$.
    
    We can then make sense of the correlation functions using a regularization procedure. Namely, for a suitable functional $F$, we can first consider the following quantity at the regularized level~\footnote{We choose not to renormalize the correlation functions using the partition function of the GFF as is usually done (e.g. in \cite{CH_construction}), since these constant quantities are not relevant for the present exposition.}:
    \begin{equation} \label{eq:reg correl}
        \begin{split}
            &\ps{F[\Phi]\prod_{k=1}^NV_{\alpha_k}(z_k)\prod_{l=1}^MV_{\beta_l}(s_l)}_{\delta,\eps,\rho} := \int_{\R^2} e^{-\ps{Q,\bm{c}}}\E\left[ F[\Phi_\rho]\prod_{k=1}^NV_{\alpha_k,\rho}(z_k)\prod_{l=1}^MV_{\beta_l,\rho}(s_l)\right.\\
            &\left.\times\exp\left( -\sum_{i=1}^2 \mu_{B,i} \int_{\Heps} \rho^{\gamma^2} e^{\ps{\gamma e_i,\Phi_\rho(x)}} dxd\bar{x} + \int_{\Reps} \rho^{\frac{\gamma^2}{2}} e^{\ps{\frac{\gamma}{2} e_i,\Phi_\rho(x)}} \mu_i(dx)\right)\right]d\bm{c}
        \end{split}
    \end{equation}
    where
    \begin{equation*}
        V_{\alpha_k,\rho}(z_k) \coloneqq \rho^{\frac{|\alpha_k|^2}{2}} e^{\ps{\alpha_k,\Phi_\rho(z_k)}} \qt{and}V_{\beta_l,\rho}(s_l) \coloneqq \rho^{\frac{|\beta_l|^2}{4}} e^{\ps{\frac{\beta_l}{2},\Phi_\rho(s_l)}}. 
    \end{equation*}
    Here the domains of integration are defined by \begin{equation*}
    \Heps = \left(\H + i\delta\right) \setminus \bigcup_{k=1}^N B(z_k,\eps)\qt{and}\Reps = \R \setminus \bigcup_{l=1}^M (s_l-\eps,s_l+\eps),
    \end{equation*}
    with $\eps$ and $\delta$ such that all the $B(z_k,\eps)$ lie in $\Heps$ and are disjoint, and all the intervals $(s_l-\eps,s_l+\eps)$ are disjoint. (see Figure \ref{fig:domains}). 
    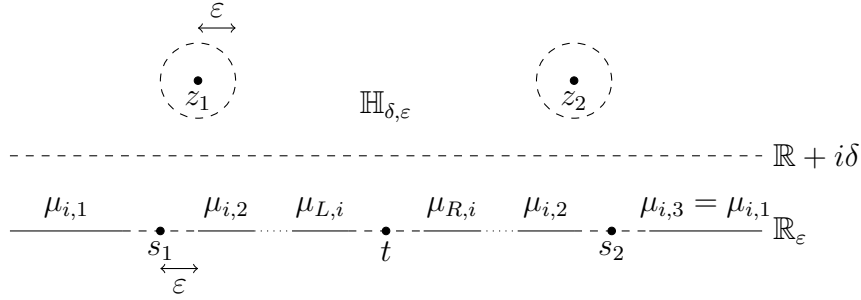
\begin{figure}
    \centering
    \begin{tikzpicture}
        \draw (-5,0) -- (-3.5,0) (-2.5,0) -- (-1.75,0) (-1.25,0) -- (-0.5,0) (0.5,0) -- (1.25,0) (1.75,0) -- (2.5,0) (3.5,0) -- (5,0) node [anchor=west] {$\mathbb{R}_\varepsilon$};
        \draw[dashed] (-3.5,0) -- (-2.5,0) (-0.5,0) -- (0.5,0) (2.5,0) -- (3.5,0) (-5,1) -- (5,1) node [anchor=west] {$\mathbb{R}+i\delta$};
        \draw [dotted] (-1.75,0) -- (-1.25,0) (1.25,0) -- (1.75,0);
        \foreach \Point/\PointLabel in {(-3,0)/s_1, (0,0)/t, (3,0)/s_2, (-2.5,2)/z_1, (2.5,2)/z_2}
        \draw [fill=black] \Point circle (0.05) node [anchor=north] {$\PointLabel$};
        \draw [<->] (-3,-0.5) -- (-2.5,-0.5) node [anchor=north east] {$\varepsilon$};
        \foreach \Point in {(-2.5,2), (2.5,2)}
        \draw [dashed] \Point circle (0.5);
        \draw [<->] (-2.5,2.7) -- (-2,2.7) node [anchor=south east] {$\varepsilon$};
        \foreach \Point/\PointLabel in {(-4.25,0)/\mu_{i,1}, (-2.125,0)/\mu_{i,2}, (-0.875,0)/\mu_{L,i}, (4.25,0)/\mu_{i,3}=\mu_{i,1}, (2.125,0)/\mu_{i,2}, (0.875,0)/\mu_{R,i}}
        \draw \Point node [anchor=south] {$\PointLabel$};
        \draw (0,2) node [anchor=north] {$\mathbb{H}_{\delta,\varepsilon}$};
    \end{tikzpicture}
    \caption{The domains $\mathbb{H}_{\delta,\varepsilon}$ and $\mathbb{R}_\varepsilon$}
    \label{fig:domains}
\end{figure}
    Note that applying a Girsanov transform (or Cameron-Martin theorem) allows one to write the regularized correlation function in the following form:
    \begin{equation} \label{eq:reg correl shifted}
        \begin{split}
            &\ps{F[\Phi]\prod_{k=1}^NV_{\alpha_k}(z_k)\prod_{l=1}^MV_{\beta_l}(s_l)}_{\delta,\eps,\rho} = C(\bm{z},\boldsymbol{\alpha})\int_{\R^2} e^{\ps{\bm{s},\bm{c}}}\E\left[ F[\Phi_\rho + H]\right.\\
            &\left.\times\exp\left( -\sum_{i=1}^2 \mu_{B,i} e^{\ps{\gamma e_i,\bm{c}}}\int_{\Heps} Z_{i,\rho}\rho^{\gamma^2} e^{\ps{\gamma e_i,\X_\rho(x)}} dxd\bar{x} + e^{\ps{\frac{\gamma}{2}e_i,\bm{c}}}\int_{\Reps} Z^\partial_{i,\rho}\rho^{\frac{\gamma^2}{2}} e^{\ps{\frac{\gamma}{2} e_i,\X_\rho(x)}} \mu_i(dx)\right)\right]d\bm{c} (1+o(1))
        \end{split}
    \end{equation}
    where $\bm{s} \coloneqq \sum \alpha_k + \frac12 \sum \beta_l -Q$, and where we have set:
    $$
    C(\bm{z},\boldsymbol{\alpha}) = \prod_{k< l} |z_k-z_l|^{-\ps{\alpha_k,\alpha_l}} \prod_{k=1}^M |z_k-\bar{z}_k|^{\frac{|\alpha_k|^2}{2}},\quad 
    H = \sum_{k=1}^{2N+M} \alpha_k G_{\hat{\C}}(\cdot,z_k),
    $$
    $$
    Z_{i,\rho} = \prod_{k=1}^{2N+M} \left( \frac{|x|_+}{|z_k-x|}\right)^{\ps{\gamma e_i,\alpha_k}} \qt{and} Z^\partial_{i,\rho} = \prod_{k=1}^{2N+M} \left( \frac{|x|_+}{|z_k-x|}\right)^{\ps{\frac{\gamma}{2} e_i,\alpha_k}}.
    $$
    
    If we choose $F=1$, then following \cite{CH_construction} we can show that under certain assumptions on the weights (the so-called \emph{Seiberg bounds}), the regularized correlation functions converge as $\rho$, $\eps$ and then $\delta$ tend to $0$, and are non-trivial in the limit. The difference here is that the boundary cosmological constants are no longer real but are rather chosen in $\C$, but this issue is addressed in \cite{Cer_HEM}. Eventually, one shows that the limiting correlation functions exist and are non-trivial if and only if the following conditions hold for all $i=1,2$:
	\begin{equation*} 
		\ps{\bm{s},\omega_i} > 0\ ;\ \ps{\alpha_k - Q,e_i} < 0 \quad\text{ for all } k=1,...,2N+M;
	\end{equation*}
    which corresponds to the Seiberg bounds~\cite[Theorem 4.6]{CH_construction}, supplemented by
    \begin{equation*} 
		\mu_{B,i} > 0 \text{ and } \Re \ \mu_{i,l} \ge 0 \text{ for all } l=1,...,M.
	\end{equation*}
    We denote by $\mc A_{N,M}$ the set of such weights satisfying the Seiberg bounds.
    We define similarly the limiting correlation function for general $F$ by the same procedure as soon as it makes sense, and denote
    $$
    \ps{F(\Phi)\prod_{k=1}^NV_{\alpha_k}(z_k)\prod_{l=1}^MV_{\beta_l}(s_l)} := \lim_{\delta\to0}\lim_{\eps\to0}\lim_{\rho\to0} \ps{F(\Phi)\prod_{k=1}^NV_{\alpha_k}(z_k)\prod_{l=1}^MV_{\beta_l}(s_l)}_{\delta,\eps,\rho}.
    $$
    The correlation functions obey the following KPZ identity, which is key to get rid of the metric dependent terms, when performing Gaussian integration by parts, as it will be seen below. 
    \begin{lemma}[KPZ identity]\label{lemma:KPZ}
        For $(\alpha_k,\beta_l)$ in $\mc A_{N,M}$, it holds that
        \begin{equation*}
        \begin{split}
            \bm{s}\ps{\prod_{k=1}^NV_{\alpha_k}(z_k)\prod_{l=1}^MV_{\beta_l}(s_l)}_{\delta,\eps,\rho} &= \sum_{i=1}^2 \left( \mu_{B,i} \gamma e_i \int_{\Heps} \ps{V_{\gamma e_i}(x)\prod_{k=1}^NV_{\alpha_k}(z_k)\prod_{l=1}^MV_{\beta_l}(s_l)}_{\delta,\eps,\rho}dxd\bar{x}\right. \\
            &\left.+ \frac{\gamma e_i}{2} \int_{\Reps} \ps{V_{\gamma e_i}(x)\prod_{k=1}^NV_{\alpha_k}(z_k)\prod_{l=1}^MV_{\beta_l}(s_l)}_{\delta,\eps,\rho} \mu_i(dx)\right).
            \end{split}
        \end{equation*}
    \end{lemma}
    \begin{proof}
        In the definition of the correlation function \eqref{eq:reg correl}, make the change of variable $\bm{c} \to \bm{c} + \frac{\boldsymbol{\rho}}{\gamma } \log \mu_{B,1}$ to obtain
        \begin{equation*} 
        \begin{split}
            &\ps{\prod_{k=1}^NV_{\alpha_k}(z_k)\prod_{l=1}^MV_{\beta_l}(s_l)}_{\delta,\eps,\rho} = \mu_{B,1}^{-\frac{\ps{\bm{s},\boldsymbol{\rho}}}{\gamma}}\int_{\R^2} e^{-\ps{\bm{s},\bm{c}}}\E\left[ \prod_{k=1}^N\tilde{V}_{\alpha_k,\rho}(z_k)\prod_{l=1}^M\tilde{V}_{\beta_l,\rho}(s_l)\right.\\
            &\left.\times\exp\left( -\sum_{i=1}^2 \frac{\mu_{B,i}}{\mu_{B,1}} \int_{\Heps} \rho^{\gamma^2} e^{\ps{\gamma e_i,\Phi_\rho(x)}} dxd\bar{x} + \mu_{B,1}^{-\frac12}\int_{\Reps} \rho^{\frac{\gamma^2}{2}} e^{\ps{\frac{\gamma}{2} e_i,\Phi_\rho(x)}} \mu_i(dx)\right)\right]d\bm{c}
        \end{split}
        \end{equation*}
        where the notation $\tilde{V}$ refers to the Vertex Operator associated to the field $\Phi_\rho -\bm{c}$ (that is, without the zero-mode). Now, differentiating in $\mu_{B,1}$ and performing the opposite shift in the zero-mode yields
        \begin{equation*} 
        \begin{split}
            \frac{\partial}{\partial\mu_{B,1}}\ps{\prod_{k=1}^NV_{\alpha_k}(z_k)\prod_{l=1}^MV_{\beta_l}(s_l)}_{\delta,\eps,\rho} &= -\frac{\ps{\bm{s},\boldsymbol{\rho}}}{\gamma \mu_{B,1}} \ps{\prod_{k=1}^NV_{\alpha_k}(z_k)\prod_{l=1}^MV_{\beta_l}(s_l)}_{\delta,\eps,\rho} \\
            &+ \frac{\mu_{B,2}}{\mu_{B,1}} \int_{\Heps} \ps{V_{\gamma e_i}(x)\prod_{k=1}^NV_{\alpha_k}(z_k)\prod_{l=1}^MV_{\beta_l}(s_l)}_{\delta,\eps,\rho}  dxd\bar{x} \\
            &+ \sum_{i=1}^2\frac{1}{2\mu_{B,1}} \int_{\Reps} \ps{V_{\gamma e_i}(x)\prod_{k=1}^NV_{\alpha_k}(z_k)\prod_{l=1}^MV_{\beta_l}(s_l)}_{\delta,\eps,\rho}  \mu_i(dx).
        \end{split}
        \end{equation*}
        On the other hand, by differentiating directly the correlation function in $\mu_{B,1}$ one obtains
        \begin{equation*} 
            \frac{\partial}{\partial\mu_{B,1}}\ps{\prod_{k=1}^NV_{\alpha_k}(z_k)\prod_{l=1}^MV_{\beta_l}(s_l)}_{\delta,\eps,\rho} = -\ps{V_{\gamma e_1}(x)\prod_{k=1}^NV_{\alpha_k}(z_k)\prod_{l=1}^MV_{\beta_l}(s_l)}_{\delta,\eps,\rho}
        \end{equation*}
        hence the result by identification. 
    \end{proof}
    It is also worth mentioning here that the correlation functions obey the following conformal covariance property that will be used in Subsection \ref{subsec:ward_global}, and that can be found in \cite[Proposition 4.7]{CH_construction}:
    \begin{proposition}
        Let $\psi\in PSL(2,\R)$ be a M\"obius transform of the upper-half plane $\H$. Then for $\bm\alpha$ in $\mc A_{N,M}$ and a suitable functional $F$:
        \begin{equation}\label{eq:conf_cov}
            \begin{split}
            &\ps{F[\Phi\circ\psi+Q\ln\norm{\psi'}]\prod_{k=1}^NV_{\alpha_k}(\psi (z_k))\prod_{l=1}^MV_{\beta_l}(\psi(s_l))}\\
            &=\prod_{k=1}^N\norm{\psi'(z_k)}^{-2\Delta_{\alpha_k}}\prod_{l=1}^M\norm{\psi'(s_l)}^{-\Delta_{\beta_l}}\ps{F[\Phi]\prod_{k=1}^NV_{\alpha_k}(z_k)\prod_{l=1}^MV_{\beta_l}(s_l)}.
            \end{split}
        \end{equation}
    \end{proposition}
    Finally, we introduce the shorthand 
 	\begin{equation*}
 	    \V\coloneqq\prod_{k=1}^NV_{\alpha_k}(z_k)\prod_{l=1}^MV_{\beta_l}(s_l).
 	\end{equation*}
    in order to lighten the presentation. 

    \subsubsection{Descendant fields and currents}\label{subsec:SET}
    Let us denote by $\Phi$ the Toda field. The way we define the regularized descendant fields is by means of expressions of the form
    \[
        \ps{F[\Phi]V_\beta(t)\prod_{k=1}^NV_{\alpha_k}(z_k)\prod_{l=1}^MV_{\beta_l}(s_l)}_{\delta,\eps,\rho}
    \]
    where $F[\Phi]$ is a polynomial in the derivatives of $\Phi$ evaluated at $t$. The expression of this polynomial is key as it encodes the symmetries of the model, and is derived from the currents $\SET$ and $\Wb$ whose modes generate the corresponding representation of the $W_3$ algebra. Before providing an explicit expression for the descendants recall that the currents are defined by
    \begin{equation*}
        \SET(z)[\Phi]=\ps{Q,\partial^2\Phi(z)}-\ps{\partial\Phi(z),\partial\Phi(z)}
    \end{equation*}
    for the stress-energy tensor, while for the higher-spin current
    \begin{equation*}
        \Wb(z)[\Phi]=q^2h_2(\partial^3\Phi(z))-2qB(\partial^2\Phi(z),\partial\Phi(z))-8C(\partial\Phi(z),\partial\Phi(z),\partial\Phi(z)).
    \end{equation*}
    Here $z$ stands for a bulk point inside $\mathbb H$ or for a boundary insertion on $\R$, in which case the complex derivative is replaced by the usual (real) derivative.

    In the framework of Vertex Operators Algebras, the descendants of the Vertex Operators $V_\alpha$ can be found using the Operator Product Expansion between the above currents and $V_\alpha$, in the sense that we have a formal expansion as $z\to w$:
    \begin{equation*}
        \Wb(z)V_\alpha(w)=\sum_{n\geq 0}(z-w)^{n-3}\Wb_{-n}V_\alpha(w)
    \end{equation*}
    and likewise for $\SET$.
    The translation between the two languages has been worked out \textit{e.g.} in~\cite{Cer_VOA}; informally speaking in order to obtain a probabilistic interpretation of the above expansion we first shift $\Phi(z)$ to $\Phi_\alpha=\Phi+\alpha\ln\frac{1}{\norm{\cdot-w}}$ and then formally Taylor expand the field $\Phi(z)$ as $\sum_{n\geq 0}\frac{(z-w)^n}{n!}\partial^n\Phi(w)$. By doing so we arrive at the following expressions for the descendant fields associated to the stress-energy tensor:
    \begin{equation}\label{eq:vir_desc}
        \begin{split}
            &\L_{-n}V_\alpha[\Phi]\coloneqq \L_{-n}^\alpha[\Phi]V_\alpha[\Phi],\qt{with}\L_{-1}^\alpha[\Phi]\coloneqq \ps{\alpha,\partial\Phi}\quad\qt{while for $n\geq 2$}\\
            &\L_{-n}^\alpha[\Phi]=\left\langle(n-1)Q+\alpha,\frac{\partial^n\Phi}{(n-1)!}\right\rangle-\sum_{i=0}^{n-2}\left\langle\frac{\partial^{i+1}\Phi}{i!},\frac{\partial^{n-i-1}\Phi}{(n-2-i)!}\right\rangle
        \end{split}
    \end{equation}
    The expression of the descendants associated to the higher-spin current is slightly more involved:
    \begin{equation}\label{eq:W_desc}
        \begin{split}
            &\Wb_{-n}V_\alpha[\Phi]\coloneqq \Wb_{-n}^\alpha[\Phi]V_\alpha[\Phi],\qt{with}\Wb_{-1}^\alpha[\Phi]\coloneqq -qB(\alpha,\partial\Phi)-2C(\alpha,\alpha,\partial\Phi),\\
            &\Wb_{-2}^\alpha[\Phi]\coloneqq q\left(B(\partial^2\Phi,\alpha)-B(\alpha,\partial^2\Phi)\right)-2C(\alpha,\alpha,\partial^2\Phi)+2C^\sigma(\alpha,\partial\Phi,\partial\Phi)\\
            &\Wb_{-n}^\alpha[\Phi]\coloneqq (n-1)(n-2)q^2\ps{h_2,\Phi^{(n)}}+q\left((n-1)B(\Phi^{(n)},\alpha)-B(\alpha,\Phi^{(n)}\right)\\
		&\left.-2C(\alpha,\alpha,\Phi^{(n)})+\sum_{i=0}^{n-2}\left(4C(\Phi^{(i+1)},\Phi^{(n-i-1)},\alpha)-2i qB(\Phi^{(i+1)},\Phi^{(n-1-i)})\right)\right.\\
		&- \frac{8}{3}\sum_{i=0}^{n-3}\sum_{j=0}^{i}C(\Phi^{(j+1)},\Phi^{(i-j+1)},\Phi^{(n-i-2)})
        \end{split}
    \end{equation}
    where we have set $\Phi^{(n)}(z)\coloneqq \frac{\partial^n\Phi(z)}{(n-1)!}$ and recall that $C^\sigma(u,v,w)=C(u,v,w)+C(u,w,v)$.

    \subsubsection{Regularized descendant fields}
    Products of derivatives of the Toda field do not make sense as such since the GFF only exists in the distributional sense. However using its regularization $\Phi_\rho$ as done above we can actually make sense of its derivatives, while for the product we will rely on Wick products instead. To be more specific, we set for positive $\rho$
    \begin{eqs}
		&\L_{-n}V_\alpha(z)[\Phi_\rho]\coloneqq \quad :\L_{-n}^\alpha[\Phi]:\ V_{\alpha,\rho}[\Phi](z)\qt{where}\\
        &:\L_{-n}^\alpha[\Phi]: \quad \coloneqq \left\langle(n-1)Q+\alpha,\frac{\partial^n\Phi}{(n-1)!}\right\rangle-\sum_{i=0}^{n-2}:\left\langle\frac{\partial^{i+1}\Phi}{i!},\frac{\partial^{n-i-1}\Phi}{(n-2-i)!}\right\rangle:
	\end{eqs}
    with $:\cdot :$ denoting the Wick product of the Gaussian random variables involved (see~\cite[Chapter III]{Janson} for more details). For instance for centered Gaussian variables we have
	\[
	:\xi_1\xi_2:=\xi_1\xi_2-\expect{\xi_1\xi_2},\qt{and}:\xi_1\xi_2\xi_3:=\xi_1\xi_2\xi_3-\xi_1\expect{\xi_2\xi_3}-\xi_2\expect{\xi_3\xi_1}-\xi_3\expect{\xi_1\xi_2}.
	\] 
    
    The regularized descendant fields are then defined with correlation functions by considering
    \begin{equation*}
        \ps{\L_{-n}V_\beta(t)\prod_{k=1}^NV_{\alpha_k}(z_k)\prod_{l=1}^MV_{\beta_l}(s_l)}_{\delta,\eps,\rho}.
    \end{equation*}
    Of course the same procedure remains valid when we consider the $\Wb$ descendant fields instead of the Virasoro ones.

	\subsection{Method for defining the descendant fields}
	Our goal is to define the descendant fields associated with the (boundary) vertex operators $V_\beta$, which we will denote $\L_{-n}V_\beta$ or $\Wb_{-m}V_\beta$. To do so we will define expressions of the form 
	\begin{equation*}
		\ps{\L_{-n}V_\beta(t)\prod_{k=1}^NV_{\alpha_k}(z_k)\prod_{l=1}^MV_{\beta_l}(s_l)}\qt{and}\ps{\Wb_{-m}V_\beta(t)\prod_{k=1}^NV_{\alpha_k}(z_k)\prod_{l=1}^MV_{\beta_l}(s_l)}
	\end{equation*}
	where the weights are chosen in such a way that the correlation function\\ $\ps{V_\beta(t)\prod_{k=1}^NV_{\alpha_k}(z_k)\prod_{l=1}^MV_{\beta_l}(s_l)}$ is well-defined, that is $(\beta,\bm\alpha)\in\mc A_{N,M+1}$.
	
	\subsubsection{Gaussian integration by parts}
	In order to make sense of such quantities we will first define them at the regularized level, that is for fixed $\delta,\eps,\rho>0$ we look at the well-defined quantities
	\begin{equation*}
		\ps{\L_{-n}V_\beta(t)\prod_{k=1}^NV_{\alpha_k}(z_k)\prod_{l=1}^MV_{\beta_l}(s_l)}_{\delta,\eps,\rho}\qt{and}\ps{\Wb_{-m}V_\beta(t)\prod_{k=1}^NV_{\alpha_k}(z_k)\prod_{l=1}^MV_{\beta_l}(s_l)}_{\delta,\eps,\rho}.
	\end{equation*}
	For this purpose we will make use of the following Gaussian integration by parts formula to provide an explicit expression for such quantities:
	\begin{lemma}\label{lemma:GaussianIPP}
		Assume that $\bm\alpha$ belongs to $\mc A_{N,M}$. Then for any $t\in\R\setminus\{s_1,\cdots,s_M\}$, $p$ a positive integer and $u\in\R^2$:
		\begin{equation*}
			\begin{split}
				&\lim\limits_{\rho\to0}\Big\langle\frac{\ps{u,\partial^p\Phi(t)}}{(p-1)!}\prod_{k=1}^NV_{\alpha_k}(z_k)\prod_{l=1}^MV_{\beta_l}(s_l)\Big\rangle_{\rho,\eps,\delta}=\sum_{k=1}^{2N+M}\frac{\ps{u,\alpha_k}}{2(z_k-t)^p}\ps{\V}_{\eps,\delta}\\
				& -\sum_{i=1}^2\int_{\Reps}\frac{\ps{u,\gamma e_i}}{2(x-t)^p}\ps{V_{\gamma e_i}(x)\V}_{\eps,\delta}\mu_i(dx)+\mu_{B,i}
				\int_{\Heps}\left(\frac{\ps{u,\gamma e_i}}{2(x-t)^p}+\frac{\ps{u,\gamma e_i}}{2(\bar x-t)^p}\right)\ps{V_{\gamma e_i}(x)\V}_{\eps,\delta}d^2x.
			\end{split}
		\end{equation*}
	\end{lemma}
	\begin{proof}
		This statement can be adapted from~\cite[Lemma 2.6]{Cer_HEM}, the only difference being that weights now belong to $\R^2$ rather than $\R$, but all the arguments employed there readily applies in this setting by using the KPZ identity from Lemma~\ref{lemma:KPZ} (in the bulk case this is done in~\cite[Lemma 3.3]{Toda_OPEWV}).
	\end{proof}
	
	Lemma~\ref{lemma:GaussianIPP} can be generalized in a straightforward way to the following equality for $p_1,\cdots,p_m$ positive integers, using Wick products as considered above:
	\begin{equation}\label{eq:IPP_product}
		\begin{split}
			&\lim\limits_{\rho\to0}\frac{1}{(p_1-1)!}\Big\langle:\prod_{l=1}^m\ps{u_l,\partial^{p_l}\Phi(t)}:\prod_{k=1}^NV_{\alpha_k}(z_k)\prod_{l=1}^MV_{\beta_l}(s_l)\Big\rangle_{\rho,\eps,\delta}\\
			&=\sum_{k=1}^{2N+M}\frac{\ps{u_1,\alpha_k}}{2(z_k-t)^{p_1}}\ps{:\prod_{l=2}^m\ps{u_l,\partial^{p_l}}\Phi(t):\V}_{\eps,\delta}\\
			& -\sum_{i=1}^2\mu_{B,i}
			\int_{\Heps}\left(\frac{\ps{u_1,\gamma e_i}}{2(x-t)^{p_1}}+\frac{\ps{u_1,\gamma e_i}}{2(\bar x-t)^{p_1}}\right)\ps{:\prod_{l=2}^m\ps{u_l,\partial^{p_l}\Phi(t)}:V_{\gamma e_i}(x)\V}_{\eps,\delta}dxd\bar x\\
			&-\sum_{i=1}^2
			\int_{\Reps}\frac{\ps{u_1,\gamma e_i}}{2(x-t)^{p_1}}\ps{:\prod_{l=2}^m\ps{u_l,\partial^{p_l}\Phi(t)}:V_{\gamma e_i}(x)\V}_{\eps,\delta}\mu_{i}(dx).
		\end{split}
	\end{equation}
	
	\subsubsection{Identification of the singular terms}
	As explained above the expression of the regularized correlation functions with descendant fields involve integrals over the domains $\Heps$ and $\Reps$, where the integrand is given by regularized correlation functions against some explicit rational functions. 
	Our goal is to show that such expressions can be written as a sum of two terms, the first one having a well-defined limit that depends analytically in the weights $\bm\alpha$ (which will be called $(P)$-class below), while the second one may diverge depending on the value of $\beta$: this second term will be a \lq\lq remainder term". Our goal is to identify this remainder term and define the descendant field as the limit of the $(P)$-class term in the two-terms expansion of the regularized correlation functions.  
	
	In order to identify this remainder term we will split the integrals coming from Gaussian integration by parts between regular and singular parts, that is away and around the insertion $t$, and introduce in this prospect the shorthands
	\begin{equation*}\label{eq:def_It}
		\begin{split}
			&\It F(t;\cdot)\coloneqq \sum_{i=1,2}\int_{\R_\eps}F^{i}(t;s)\mu_i(ds)+\mu_{B,i}\int_{\Heps}\left(F^i(t;x)+F^i(t;\bar x)\right)d^2x,\\
			&\Is F(t;\cdot)\coloneqq \sum_{i=1,2}\int_{\R_\eps^1}F^{i}(t;s)\mu_i(ds)+\mu_{B,i}\int_{\Heps^1}\left(F^i(t;x)+F^i(t;\bar x)\right)d^2x\qt{and}\\
			&\Ir F(t;\cdot)\coloneqq \sum_{i=1,2}\int_{\R_\eps^c}F^{i}(t;s)\mu_i(ds)+\mu_{B,i}\int_{\Heps^c}\left(F^i(t;x)+F^i(t;\bar x)\right)d^2x
		\end{split}
	\end{equation*}
	with $\Reps^1\coloneqq\Reps\cap B(t,r)$ for $r>0$ small enough so that all the insertions that appear in the above correlation functions are at distance at least $2r$ from $t$. We likewise denote $\Heps^1\coloneqq \Heps\cap B(t,r)$ and $\Heps^c\coloneqq\Heps\setminus\Heps^1$. 
	Here the notation $F(t;\cdot)$ indicate a pair of functions $(F^1,F^2)$; they should be thought of as linear combinations of expressions of the form $\frac{1}{(t-\cdot)^{n}}\ps{V_{\gamma e_i}(\cdot)V_\beta(t)\V}_{\delta,\eps,\rho}$ that arise in the Gaussian integration by parts. 
 
    More generally since the Gaussian integration by parts will involve multiple integrals we can analogously define for tuples of functions $(F^{i_1,\cdots,i_p}(t;x_1,\cdots,x_p))_{\bm i\in\{1,2\}^p}$ the notation
	\begin{equation*}
		\begin{split}
			\Is^p F(t;\cdot)\coloneqq \sum_{\bm i \in\{1,2\}^p}\sum_{\substack{k_1,k_2,k_3\geq0\\ k_1+k_2+k_3=p}}&\int_{\mc R_{k_1}\times\mc H_{k_2,k_3}}\prod_{l=1}^{k_1}\mu_{i_l}(ds_l)\prod_{l=k_1+1}^{p} \mu_{B,i_l}d^2x_l\\
			&F^{\bm i}(t;s_1,\cdots,s_{k_1},x_{k_1+1}\cdots,x_{k_1+k_2},\bar x_{k_1+k_2+1},\cdots,\bar x_p)
		\end{split}
	\end{equation*}
    where $\mc R_{k_1}=\left(\Reps\cap B(t,r_1)\right)\times\cdots\times\left(\Reps\cap B(t,r_{k_1})\right)$ for $r_1>\cdots>r_{k_1}>0$ and likewise for $\mc H_{k_2+k_3}$.
	More generally with an obvious adaptation we can analogously define the quantities 
	\[
	\It^{p_1}\times\Is^{p_2}\times\Ir^{p_3}F(t;\cdot)    
	\]
	where the integrand is the same as above but the first $p_1$ integrals are taken over either $\Reps$ or $\Heps$, while the following $p_2$ ones are evaluated around the insertion $t$ and the last $p_3$ ones are defined away from $t$, where the different radii chosen will be made explicit.
	
	\subsubsection{Symmetrization identities and Stoke's formula}
	In order to make explicit the singular behaviour of such integrals and discard \lq\lq fake singularities" we will rely on \lq\lq symmetrization identites", that is equalities for non-coinciding points
	\[
	\sum_{i=1}^{n-1}\frac{1}{(x-t)^i(y-t)^{n-i}}=\frac{1}{x-y}\left(\frac{1}{(y-t)^{n-1}}-\frac{1}{(x-t)^{n-1}}\right),
	\]
	that we use to simplify the expression of the regularized correlation functions with descendants. This will allow us to rewrite such quantities in term of integrals of \textit{total derivatives}. More precisely we will see that the remainder terms coming from Gaussian integration by parts can be written as a linear combination of expressions of the form
	\begin{equation}\label{eq:blablabla}
		\mathfrak R\coloneqq\Is^{p_1}\times \Ir^{p_2}\times\It^{p_3}\left[\partial_{x_1}\cdots\partial_{x_l}\left(F(t,\bm x)\ps{V_{\gamma e_{i_1}}(x_1)\cdots V_{\gamma e_{i_p}}(x_p)V_\beta(t)\V}_{\delta,\eps}\right)\right].
	\end{equation}
	These can be transformed using Stoke's formula: namely we can write that
    \begin{align*}
        &\Is\left[\partial_{x_1}\left(F(t,x)\ps{V_{\gamma e_i}(x)V_\beta(t)\V}_{\delta,\eps}\right)\right]=\\
        &\sum_{i=1}^2\mu_{L,i_1}\left[F(t,x)\ps{V_{\gamma e_i}(x)V_\beta(t)\V}_{\delta,\eps}\right]_{t-r}^{t-\eps}+\mu_{R,i_1}\left[F(t,x)\ps{V_{\gamma e_i}(x)V_\beta(t)\V}_{\delta,\eps}\right]^{t+r}_{t+\eps}\\
        &+\mu_{B,i}\int_{(t-r,t+r)}\left(F(t,x+i\delta)-F(t,x-i\delta)\right)\ps{V_{\gamma e_i}(x+i\delta)V_\beta(t)\V}_{\delta,\eps} \frac{idx}2\\
        &+\mu_{B,i}\oint_{\Heps\cap \partial B(t,r)}\ps{V_{\gamma e_i}(\xi)V_\beta(t)\V}_{\delta,\eps}\left(F(t,\bar\xi)\frac{d\bar\xi}2-F(t,\xi)\frac{id\xi}2\right).
    \end{align*}
    More generally applying Stoke's formula to Equation~\eqref{eq:blablabla} will yield several terms of this form and involving composite integrals.
    Such expressions will contain the remainder terms but also some convergent quantities that also depend on $r$. In the above example, the remainder terms will be given by 
    \begin{align*}
        &\sum_{i=1}^2\mu_{L,i_1}\left(F(t,t-\eps)\ps{V_{\gamma e_i}(t-\eps)V_\beta(t)\V}_{\delta,\eps}\right)-\mu_{R,i_1}\left(F(t,t+\eps)\ps{V_{\gamma e_i}(t+\eps)V_\beta(t)\V}_{\delta,\eps}\right)\\
        &+\mu_{B,i}\int_{(t-r,t+r)}\left(F(t,x+i\delta)-F(t,x-i\delta)\right)\ps{V_{\gamma e_i}(x+i\delta)V_\beta(t)\V}_{\delta,\eps} \frac{idx}2.
    \end{align*}
    One sees that for $\ps{\beta,e_i}$ negative enough this converges to $0$ as first $\eps$ and then $\delta$ go to $0$. Conversely if $F=1$ and $\ps{\beta,\gamma e_i}$ is positive then this term is seen to be divergent.
    
    To discard the other terms appearing in the expansion, for $\mathfrak R$ given by Equation~\eqref{eq:blablabla} we define $\tilde{\mathfrak R}$ to be the sum of the terms that involve at least one such remainder term. Put differently $\tilde{\mathfrak R}$ is defined from $\mathfrak R$ by removing the terms involving \textbf{only} correlation functions $\ps{\prod_{p=1}^nV_{\gamma e_{i_p}}(x_p)V_\beta(t)\V}_{\delta,\eps}$ where all the $x_p$ are at fixed distance $r_p$ from $t$. This remainder term $\tilde{\mathfrak R}$ will be of crucial importance in the definition of the descendant field.
	
	\subsubsection{Taking the limit}
	We then investigate the limit of these remainder terms as first $\rho\to0$ and then $\eps$ and $\delta$ go to $0$. As we will see (and for generic value of $\beta$) they will either vanish in the limit or diverge; in the latter case we need to take this into account when defining the correlation function $\ps{\L_{-n}V_\beta(t)\prod_{k=1}^NV_{\alpha_k}(z_k)\prod_{l=1}^MV_{\beta_l}(s_l)}$. Namely if we denote by $\tilde{\mathfrak{L}}_{-n,\delta,\eps,\rho}(\bm\alpha)$ such remainder terms the above correlation function will be defined as the limit
	\begin{equation*}
		\begin{split}
			&\ps{\L_{-n}V_\beta(t)\prod_{k=1}^NV_{\alpha_k}(z_k)\prod_{l=1}^MV_{\beta_l}(s_l)}\coloneqq\\
			&\lim\limits_{\delta,\eps,\rho\to0}\ps{\L_{-n}V_\beta(t)\prod_{k=1}^NV_{\alpha_k}(z_k)\prod_{l=1}^MV_{\beta_l}(s_l)}_{\delta,\eps,\rho}-\tilde{\mathfrak{L}}_{-n,\delta,\eps,\rho}(\bm\alpha).
		\end{split}
	\end{equation*}
	For instance we will see that
	\begin{equation*}
		\begin{split}
			&\tilde{\mathfrak{L}}_{-1,\delta,\eps,\rho}(\bm\alpha)=\sum_{i=1}^2\left(\mu_{L,i}\ps{V_{\gamma e_i}(t-\eps)\V}_{\delta,\eps}-\mu_{R,i}\ps{V_{\gamma e_i}(t+\eps)\V}_{\delta,\eps}\right)
		\end{split}
	\end{equation*} 
	so that the $\L_{-1}V_\beta$ descendant is defined within correlation functions by setting
	\begin{equation*}
		\begin{split}
			&\ps{\L_{-1}V_\beta(t)\prod_{k=1}^NV_{\alpha_k}(z_k)\prod_{l=1}^MV_{\beta_l}(s_l)}\coloneqq\lim\limits_{\delta,\eps,\rho\to0}\\
			&\ps{\L_{-1}V_\beta(t)\V}_{\delta,\eps,\rho}-\sum_{i=1}^2\Big(\mu_{L,i}\ps{V_{\gamma e_i}(t-\eps)V_\beta(t)\V}_{\delta,\eps}-\mu_{R,i}\ps{V_{\gamma e_i}(t+\eps)V_\beta(t)\V}_{\delta,\eps}\Big).
		\end{split}
	\end{equation*}
	
	The same applies for other Virasoro descendants, but the expressions for the remainder terms get increasingly involved. And likewise, we will proceed in the same way to make sense the $W$-descendants, in which case their definitions rely on the very same method but involve some additional combinatorial identities due to the higher-spin symmetry.

    \subsubsection{Some notations}
	In order to lighten the computations to come we introduce several shortcuts. First of all recall the definition of the symbols $\It$, $\Is$  and $\Ir$ from Equation~\eqref{eq:def_It}. We will also frequently use the notation $\V\coloneqq\prod_{k=1}^NV_{\alpha_k}(z_k)\prod_{l=1}^MV_{\beta_l}(s_l)$ to shorten the computations.  As explained in the Gaussian integration by parts formula from Equation~\eqref{eq:IPP_product}, when defining the descendants we will consider correlation functions containing additional insertion $V_{\gamma e_i}$. To simplify the treatment of these quantities we therefore set for indices $i_k\in\{1,2\}$ and distinct insertions $x_k\in \H\cup\R$
    \begin{equation}\label{eq:def_Psi}
        \Psi_{i_1,\cdots,i_l}(x_1,\cdots,x_l)\coloneqq \ps{V_{\gamma e_{i_1}(x_1)}\cdots V_{\gamma e_{i_l}(x_l)}V_\beta(t)\V}_{\rho,\eps,\delta}
    \end{equation}
    where for the sake of simplicity we omit the regularization indices on the left-hand side.

    \subsection{Some technical estimates}
    In order to implement the previous method we will need to ensure that the quantities are well-defined and for this it is necessary to understand several of their analytic properties. We gather such properties in this subsection. Since the proofs of these statements are very similar to the case of boundary Liouville CFT (for the fusion estimates~\cite{fusion} while for analycity we refer to~\cite{ARSZ}) and of bulk Toda CFT (fusion estimates can be found in~\cite[Lemma 3.2]{Toda_OPEWV} and analycity in~\cite[Section 2.2]{Toda_correl2}), and proceed in the exact same way, we do not include them in the present document.
    
	\subsubsection{Bounds at infinity}
    As explained above, when implementing our method we will need to consider integrals over the upper-half plane or the real line that contain correlation functions. These correlation functions feature additional Vertex Operators of the form $V_{\gamma e_i}$ for some $i=1,2$.  In order to ensure that such integrals are indeed well-defined we first need to understand what happens when these insertions diverge. To this end we provide the following statement which provides bounds on the correlation functions when their arguments go to $\infty$:
    \begin{lemma}\label{lemma:inf_integrability_toda}
		Assume that $\bm{\alpha}\in\mc A_{N,M}$ and consider for $i=1,2$ additional insertions $\bm{x}_i \coloneqq\left(x^{(1)}_i,\cdots,x^{(N_i)}_i\right)\in\H^{N_i}$ and $\bm{y}_i \coloneqq\left(y^{(1)}_i,\cdots,y^{(M_i)}_i\right)\in\R^{M_i}$. Then for any  $h>0$, if for $i=1,2$ the $\bm{x}_i$, $\bm{y}_i$ and $\bm{z}$ stay in the domain $U_h\coloneqq \left\lbrace{\bm w, h<\min\limits_{n\neq m}\norm{w^{(n)}-w^{(m)}}}\right\rbrace$, then there exists $C=C_h$ such that, uniformly in $\delta,\eps,\rho$,
		\begin{equation*}
			\ps{\prod_{i=1}^2 \prod_{n=1}^{N_i}V_{\gamma e_i}\left(x^{(n)}_i\right) \prod_{m=1}^{M_i}V_{\gamma e_i}\left(y^{(m)}_i\right)\V}_{\delta,\eps,\rho}\leq C_h \prod_{i=1}^2 \prod\limits_{n=1}^{N_i}\left(1+\norm{x^{(n)}_i}\right)^{-4}\prod\limits_{m=1}^{M_i}\left(1+\norm{y^{(m)}_i}\right)^{-2}.
		\end{equation*}
	\end{lemma}
    Thanks to this estimate we see that the correlation functions are integrable at $\infty$. This will allow us in the sequel to infer that boundary terms that may come from the successive integration by parts do indeed vanish at infinity. As such we won't consider such terms in the definition of the descendant fields in the next two sections.
	
	\subsubsection{Fusion estimates}	
	In order to ensure finiteness of the correlation functions and of their derivatives we will also need to ensure that the integrals that we encounter are not singular near the insertion points. For this purpose we rely on \textit{fusion asymptotics} of the correlation functions, that is their rate of divergence when several insertion points collide: 
	\begin{lemma}\label{lemma:fusion}
		Assume that $\bm{\alpha}\in\mc A_{N,M}$ and that all pairs of points in $\bm{z}$ are separated by some distance $h>0$ except for one pair $(z_1,z_2)$. Further assume that $\ps{\alpha_1+\alpha_2-Q,e_1}<0$. 
		Then as $z_1\to z_2$, for any positive $\eta$ there exists a positive constant $K$ such that, uniformly on $\delta,\eps,\rho$:
		\begin{enumerate}
			\item if $z_1,z_2\in\H$ then
			\begin{equation}\label{eq:fusion_hh}
				\ps{\V}_{\delta,\eps,\rho}\leq K \norm{z_1-z_2}^{-\ps{\alpha_1,\alpha_2}+\left(\frac{1}{2}\ps{\alpha_1+\alpha_2-Q,e_2}^2-\eta\right)\mathds{1}_{\ps{\alpha_1+\alpha_2-Q,e_2}>0}};
			\end{equation}
			\item if $z_1,z_2\in\R$ then
			\begin{equation}\label{eq:fusion_rr}
				\ps{\V}_{\delta,\eps,\rho}\leq K \norm{z_1-z_2}^{-\frac{\ps{\alpha_1,\alpha_2}}2+\left(\frac{1}{4}\left(\ps{\alpha_1+\alpha_2-Q,e_2}\right)^2-\eta\right)\mathds{1}_{\ps{\alpha_1+\alpha_2-Q,e_2}>0}};
			\end{equation}
			\item if $z_1\in\H$ while $z_2\in\R$ with in addition $\ps{\alpha_1-\frac Q2,e_1}<0$ then
			\begin{equation}\label{eq:fusion_hr}
            \begin{split}
                &\ps{\V}_{\delta,\eps,\rho}\leq K\times \norm{z_1-\bar z_1}^{-\frac{\norm{\alpha_1}^2}2+\left(\left(\ps{\alpha_1-\frac Q2,e_2}\right)^2-\eta\right)\mathds{1}_{\ps{\alpha_1-\frac Q2,e_2}>0}} \\
                & \norm{z_1-z_2}^{-\ps{\alpha_1,\alpha_2}+\left(\left(\ps{\alpha_1+\frac{\alpha_2}2-\frac Q2,e_2}\right)^2-\eta\right)\mathds{1}_{\ps{\alpha_1+\frac{\alpha_2}2-\frac Q2,e_2}>0}}.
            \end{split}
			\end{equation}
		\end{enumerate}
        As a consequence for $\beta$ such that $\ps{\beta,e_i}<0$ we have that $\norm{x-t}^{\frac{\ps{\beta,\gamma e_i}}2}F_i(x)$ is continuous at $x=t$, where $F_i$ is of the form
        $$
            F_i(x) = \Is^p\times\Ir^q \left[r_{i,\bm{i}}(x,\bm{x})\Psi_{i,\bm{i}}(x,\bm{x})\right]
        $$
        with $r_{i,\bm{i}}=r_{i,i_1,...,i_{(p+q)}}$ regular at $x=t$ and such that $r_{i,\bm{i}}(x,\bm{x})\Psi_{i,\bm{i}}(\bm{x})$ is integrable on $\Is^p\times\Ir^q$, while $\Psi_{i,\bm{i}}(x,\bm{x}) = \Psi_{i,i_1,...,i_{p+q}}(x,x_1,...,x_{p+q})$.
	\end{lemma}
    \begin{proof}
        The first part of the claim (fusion estimates) follows from the counterparts statements in the boundary Liouville case~\cite[Section 5]{fusion} (see also~\cite[Lemma 2.4]{Cer_HEM}) and bulk Toda theory~\cite[Lemma 3.2]{Toda_OPEWV} (this adaptation is made possible since the two GFFs $\ps{\X,e_i}$ that appear in the definitions of the GMC measures are negatively correlated since $\ps{e_1,e_2}<0$). As for the second part (continuity at $x=t$) we use the probabilistic representation to infer that $\norm{x-t}^{\frac{\ps{\beta,\gamma e_i}}2}r_{i,\bm{i}}(x,\bm{x})\Psi_{i,\bm{i}}(x,\bm{x})$ is continuous at $x=t$. To deduce continuity of $\norm{x-t}^{\frac{\ps{\beta,\gamma e_i}}2}F_i$ from this fact we then rely on the integrability properties of the correlation functions in the form of Lemma~\ref{lemma:inf_integrability_toda}.
    \end{proof}
    
    \subsubsection{A class of remainder terms}
    In the definition of the descendant fields as described above, we need to take a limit of a regularized quantity to which we have substracted a remainder term. In order for this procedure to define a meaningful object, we introduce a class of remainder terms that would satisfy the properties we expect from a definition of the descendant fields:
	\begin{defi}
		We will say that a regularized quantity $F_{\delta,\eps,\rho}(\bm\alpha)$ has the $(P)$ property or is $(P)$-class when it satisfies the following assumptions:
		\begin{enumerate}
			\item for any $\bm\alpha\in\mc A_{N,M}$, the limit $F(\bm\alpha)\coloneqq\lim\limits_{\delta,\eps,\rho\to0}F_{\delta,\eps,\rho}(\bm\alpha)$ exists and is finite as $\rho$, $\eps$ and then $\delta\to0$;
			\item the map $\bm\alpha\mapsto F(\bm\alpha)$ is analytic in a complex neighborhood of $\mc A_{N,M}$.
		\end{enumerate}
	\end{defi}
	Let us now provide some examples of such quantities. And to start with we stress that the correlation functions themselves satisfy this property:
	\begin{lemma}\label{lemma:ana_correl}
		The regularized correlation functions $\ps{\prod_{k=1}^NV_{\alpha_k}(z_k)\prod_{l=1}^MV_{\beta_l}(s_l)}_{\delta,\eps,\rho}$ are $(P)$-class.
	\end{lemma}
	Likewise some singular integrals (that will later on enter the definition of the descendants) satisfy this property:
	\begin{lemma}~\label{lemma:fusion_integrability}
		For any $i$ and $j$ in $\{1,2\}$ the following integrals are $(P)$-class:
		\begin{equation}\label{eq:fusion_int}
			\begin{split}
				&\int_{\frac12\D\times(\D\setminus\frac12\D)}\frac1{y-x}\ps{V_{\gamma e_i}(x+i)V_{\gamma e_j}(y+i)\V}_{\delta,\eps,\rho}d^2xd^2y,\\
				&\int_{-1}^0\int_0^1\frac1{y-x}\ps{V_{\gamma e_i}(x)V_{\gamma e_j}(y)\V}_{\delta,\eps,\rho}dxdy\quad\text{and}\\
				&\int_{\D\cap\H}\int_1^2\frac1{y-x}\ps{V_{\gamma e_i}(x)V_{\gamma e_j}(y)\V}_{\delta,\eps,\rho}d^2xdy.
			\end{split}	
		\end{equation}
    In particular integrals of the form 
    \begin{equation*}
        \Is\times\Ir\left[\frac{1}{x-y}\Psi_{i,j}(x,y)\right]
    \end{equation*}
    are $(P)$-class.
	\end{lemma}

%% file: proof_ward.tex
\section{Definition of the descendant fields and Ward identities}
	This section is dedicated to providing a definition of the descendant fields in the case where only one current is involved, that is to say we will provide a meaning to the quantities $\L_{-n}V_\beta$ and $\Wb_{-m}V_\beta$ for any positive integers $n,m$ and for suitable $\beta$ (that is $\beta\in Q+\wc$).	We will define such descendants for boundary Vertex Operators, the bulk case being treated via the same arguments without the technicalities that appear in the boundary case.
	
	To define such descendants we will make sense of them within correlation functions, that is we will make sense of the quantities
	\begin{equation*}
		\ps{\L_{-n}V_\beta(t)\prod_{k=1}^NV_{\alpha_k}(z_k)\prod_{l=1}^MV_{\beta_l}(s_l)}\qt{and}\ps{\Wb_{-n}V_\beta(t)\prod_{k=1}^NV_{\alpha_k}(z_k)\prod_{l=1}^MV_{\beta_l}(s_l)}
	\end{equation*}
	where $(\beta,\bm\alpha)\in\mc A_{N,M+1}$ and with $t\in\partial\H=\R$.

	\subsection{Descendant fields at the first order}
	To start with we will define the descendant associated to Virasoro and higher-spin symmetry at the level one. These are the easiest ones to deal with but their definition still relies on the type of argument that we will use in the more general case considered in this paper.	
	
	\subsubsection{Definition of the descendant fields}
	In constrast with the closed case, defining the descendant fields in the presence of a boundary requires some extra care. In particular in order to define them we need to first make sense of them at the regularized level and then take an appropriate limit as in the closed case, but in addition to that we need to substract some additional terms that may blow up and whose presence comes from the fact that the surface considered has a boundary. For this purpose let us introduce the following quantities.	
	\begin{lemma}\label{lemma:desc1}
		For any $i$ in $\{1,2\}$ set 
		\begin{equation*}
			\begin{split}
				&\mathfrak{L}^i_{-1,\delta,\eps,\rho}(\bm\alpha)\coloneqq \Is\left[\partial_x \ps{V_{\gamma e_i}(x)V_\beta(t)\V}_{\delta,\eps,\rho}\right].
			\end{split}
		\end{equation*}
		Then the following quantity is $(P)$-class:
		\begin{equation*}
			\begin{split}
				&\ps{\L_{-1}V_\beta(t)\prod_{k=1}^NV_{\alpha_k}(z_k)\prod_{l=1}^MV_{\beta_l}(s_l)}_{\delta,\eps,\rho}-\sum_{i=1}^2\mathfrak{L}^i_{-1,\delta,\eps,\rho}(\bm\alpha).
			\end{split}
		\end{equation*}
		In the same fashion, let us define for $i=1,2$ 
		\begin{equation*}
			\begin{split}
				&\mathfrak{W}_{-1,\delta,\eps,\rho}^i(\bm\alpha)\coloneqq -h_2(e_i)\left(q-2\omega_{\hat i}(\beta)\right)\mathfrak{L}^i_{-1,\delta,\eps,\rho}(\bm\alpha).
			\end{split}
		\end{equation*}
		Then the following quantity is $(P)$-class:
		\begin{equation*}
			\begin{split}
				&\ps{\Wb_{-1}V_\beta(t)\prod_{k=1}^NV_{\alpha_k}(z_k)\prod_{l=1}^MV_{\beta_l}(s_l)}_{\delta,\eps,\rho}-\sum_{i=1}^2\mathfrak{W}^i_{-1,\delta,\eps,\rho}(\bm\alpha).
			\end{split}
		\end{equation*}
	\end{lemma}
	\begin{proof}
		Fix $\delta,\eps$ to be positive and set $\Desc$ to be either $\L$ or $\Wb$.
		Using Gaussian integration by parts we have
		\begin{equation*}
			\begin{split}
				\ps{\Desc_{-1}V_\beta(t)\V}_{\delta,\eps} = &\sum_{k=1}^{2N+M} \frac{\Desc_{-1}^\beta(\alpha_k)}{2(z_k-t)} \ps{V_\beta(t)\V}_{\delta,\eps} - \sum_{i=1}^2 \int_{\R_\eps} \frac{\Desc_{-1}^\beta(\gamma e_i)}{2(s-t)} \ps{V_{\gamma e_i}(s)V_\beta(t)\V}_{\delta,\eps} \mu_{i}(ds)\\
				&-\sum_{i=1}^2\mu_{B,i} \int_{\H_{\delta,\eps}} \left( \frac{\Desc_{-1}^\beta(\gamma e_i)}{2(x-t)}+\frac{\Desc_{-1}^\beta(\gamma e_i)}{2(\bar{x}-t)}\right) \ps{V_{\gamma e_i}(x)V_\beta(t)\V}_{\delta,\eps} d^2x.
			\end{split}
		\end{equation*}
		As $\eps,\delta\to0$ the first term in this expansion is $(P)$-class in agreement with Lemma~\ref{lemma:ana_correl}.
		As for the integrals, we see that thanks to Lemma~\ref{lemma:fusion} the integrals over the subdomains $\Reps^c$ and $\Heps^c$ are $(P)$-class so that 
		\begin{align*}
			\ps{\Desc_{-1}V_\beta(t)\V}_{\delta,\eps} = \text{$(P)$-class terms}& -\Is \left[\frac{\Desc_{-1}^\beta(\gamma e_i)}{2(x-t)} \ps{V_{\gamma e_i}(x)V_\beta(t)\V}_{\delta,\eps}\right].
		\end{align*}
		To be more explicit the $(P)$-class terms in the above are explicitly given by
		\begin{align*}
			\sum_{k=1}^{2N+M} \frac{\Desc_{-1}^\beta(\alpha_k)}{2(z_k-t)} \ps{V_\beta(t)\V}_{\delta,\eps}-\Ir \left[\frac{\Desc_{-1}^\beta(\gamma e_i)}{2(x-t)} \ps{V_{\gamma e_i}(x)V_\beta(t)\V}_{\delta,\eps}\right].
		\end{align*}
		
		To treat the remaining terms we first note that we have the following explicit equalities (which can be checked based on explicit computations):
		\begin{equation*}
			\L_{-1}^\beta(\gamma e_i)=\ps{\beta,\gamma e_i},\quad\Wb_{-1}^\beta(\gamma e_i)=-h_2(e_i)\left(q-2\omega_{\hat i}(\beta)\right)\ps{\beta,\gamma e_i}.
		\end{equation*}
		In particular in both cases the ratio $\frac{\Desc_{-1}^\beta(\gamma e_i)}{\ps{\beta,\gamma e_i}}$ is well-defined for any $\beta$ in $\C^2$. This allows to rewrite the remaining integrals as
		\begin{equation*}
			\begin{split}
				\frac{\Desc_{-1}^\beta(\gamma e_i)}{\ps{\beta,\gamma e_i}}\Is \left[\frac{\ps{\beta,\gamma e_i}}{2(x-t)} \ps{V_{\gamma e_i}(x)V_\beta(t)\V}_{\delta,\eps}\right].
			\end{split}
		\end{equation*}
		We recognize there derivatives of correlation functions in that
		\begin{align*}
			\partial_x \ps{V_{\gamma e_i}(x)V_\beta(t)\V}_{\delta,\eps}=&\left(\frac{\ps{\beta,\gamma e_i}}{2(t-x)}+\frac{\ps{\gamma e_i,\gamma e_i}}{2(\bar x-x)} + \sum_{k=1}^{2N+M} \frac{\ps{\gamma e_i,\alpha_k}}{2(z_k-x)}\right)\ps{V_{\gamma e_i}(x)V_\beta(t)\V}_{\delta,\eps}\\
			&- \It \left[\frac{\ps{\gamma e_i,\gamma e_j}}{2(y-x)} \ps{V_{\gamma e_i}(x)V_{\gamma e_j}(y) V_\beta(t) \V}_{\delta,\eps}\right]
		\end{align*}
		and likewise for $\partial_{\bar x}$ and $\partial_s$, where the integral is over $y$. 
		This shows that
		\begin{align*}
			&\Is\left[\frac{\ps{\beta,\gamma e_i}}{2(x-t)}\ps{V_{\gamma e_i}(x)V_\beta(t)\V}_{\delta,\eps}\right]=\Is\left[\left(\partial_x+\sum_{k=1}^N\frac{\ps{\alpha_k,\gamma e_i}}{2(x-z_k)}\right)\ps{V_{\gamma e_i}(x)V_\beta(t)\V}_{\delta,\eps}\right]\\
			&+\sum_{i,j=1}^2\mu_{B,i}\mu_{B,j}\int_{ \Heps^1\times\Heps}\left(\frac{\ps{\gamma e_i,\gamma e_j}}{2(x-y)}+\frac{\ps{\gamma e_i,\gamma e_j}}{2(x-\bar y)}+\frac{\ps{\gamma e_i,\gamma e_j}}{2(\bar x-y)}+\frac{\ps{\gamma e_i,\gamma e_j}}{2(\bar x-\bar y)}\right)\ps{V_{\gamma e_i}(x)V_{\gamma e_j}(y)\V}_{\delta,\eps}d^2xd^2y\\
			&+\sum_{i,j=1}^2\mu_{B,i}\int_{ \Heps^1\times\Reps}\left(\frac{\ps{\gamma e_i,\gamma e_j}}{2(x-y)}+\frac{\ps{\gamma e_i,\gamma e_j}}{2(\bar x-y)}\right)\ps{V_{\gamma e_i}(x)V_{\gamma e_j}(y)\V}_{\delta,\eps}d^2x\mu_{j}(dy)\\
			&+\sum_{i,j=1}^2\mu_{B,i}\int_{ \Heps\times\Reps^1}\left(\frac{\ps{\gamma e_i,\gamma e_j}}{2(y-x)}+\frac{\ps{\gamma e_i,\gamma e_j}}{2(y-\bar x)}\right)\ps{V_{\gamma e_i}(x)V_{\gamma e_j}(y)\V}_{\delta,\eps}d^2x\mu_{j}(dy)\\
			&+\sum_{i,j=1}^2\int_{ \Reps^1\times\Reps}\frac{\ps{\gamma e_i,\gamma e_j}}{2(x-\bar y)}\ps{V_{\gamma e_i}(x)V_{\gamma e_j}(y)\V}_{\delta,\eps}\mu_{i}(dx)\mu_{j}(dy).
		\end{align*}
		Now by symmetry in the variables $x$ and $y$ the diagonal part of the integral over $\Heps\times\Heps$ (that is the subintegral over $\Heps^1\times\Heps^1$) vanishes so that only the integral over $\Heps^1\times\Heps^c$ remains. Thanks to Lemma~\ref{lemma:fusion_integrability} we know that this integral is $(P)$-class. The same argument applies to the remaining integrals, whose sum is actually non-zero only over $\Heps^1\times\Reps^c$, $\Heps^c\times\Reps^1$ and $\Reps^c\times\Reps^c$. Recollecting term we obtain that
		\begin{align*}
			\ps{\Desc_{-1}V_\beta(t)\V}_{\delta,\eps} = &\text{$(P)$-class terms} + \Is\left[\frac{\Desc_{-1}^\beta(\gamma e_i)}{\ps{\beta,\gamma e_i}}\partial_x \ps{V_{\gamma e_i}(x)V_\beta(t)\V}_{\delta,\eps}\right].
		\end{align*}
		This concludes for the proof of Lemma~\ref{lemma:desc1}.
	\end{proof}
	
	Based on this observation we now would like to define the descendant fields $\L_{-1}V_\beta$ and $\Wb_{-1}V_\beta$. However in order to remove the dependence with respect to the radius $r$ and as explained in the previous section we actually need to remove one additional term. In this perspective let us recall that we have introduced there the notations $\tilde{\mathfrak{L}}_{-1,\delta,\eps,\rho}(\bm\alpha)$ and $\tilde{\mathfrak{W}}_{-1,\delta,\eps,\rho}(\bm\alpha)$ to discard such terms. Here we see that more explicitly the remainder term is given by (up to a $o(1)$ term)
	\begin{equation*}
		\begin{split}
            &\tilde{\mathfrak{L}}_{-1,\delta,\eps,\rho}(\bm\alpha)=\sum_{i=1}^2\tilde{\mathfrak{L}}_{-1,\delta,\eps,\rho}^i(\bm\alpha)\qt{and}\tilde{\mathfrak{W}}_{-1,\delta,\eps,\rho}(\bm\alpha)=\sum_{i=1}^2\tilde{\mathfrak{W}}_{-1,\delta,\eps,\rho}^i(\bm\alpha),\text{ where}\\
			&\tilde{\mathfrak{L}}^i_{-1,\delta,\eps,\rho}(\bm\alpha)=\left(\mu_{L,i}\ps{V_{\gamma e_i}(t-\eps)\V}_{\delta,\eps}-\mu_{R,i}\ps{V_{\gamma e_i}(t+\eps)\V}_{\delta,\eps}\right)\qt{while}\\
            &\tilde{\mathfrak{W}}_{-1,\delta,\eps,\rho}^i(\bm\alpha)\coloneqq -h_2(e_i)\left(q-2\omega_{\hat i}(\beta)\right)\tilde{\mathfrak{L}}^i_{-1,\delta,\eps,\rho}(\bm\alpha).
		\end{split}
	\end{equation*}
	Note that $\mathfrak{L}^i_{-1,\delta,\eps,\rho}(\bm\alpha)$ and $\tilde{\mathfrak{L}}^i_{-1,\delta,\eps,\rho}(\bm\alpha)$ only differ by $(P)$-class terms in that:
	\begin{equation*}
		\begin{split}
			&\tilde{\mathfrak{L}}^i_{-1,\delta,\eps,\rho}(\bm\alpha)=\mathfrak{L}^i_{-1,\delta,\eps,\rho}(\bm\alpha)-\left(\mu_{R,i}\ps{V_{\gamma e_i}(t+r)\V}_{\delta,\eps}-\mu_{L,i}\ps{V_{\gamma e_i}(t-r)\V}_{\delta,\eps}\right)\\
			&-\mu_{B,i}\int_{ \Heps\cap \partial B(t,r)}\ps{V_{\gamma e_i}(\xi)\V}_{\delta,\eps}\frac{id\bar\xi-id\xi}{2}.
		\end{split}
	\end{equation*}
	
	\begin{defi}\label{def:desc1}
		Take $t\in\R$ and $(\beta,\bm\alpha)\in\mc A_{N,M+1}$.  We define the descendant field $\L_{-1}V_\beta$ within half-plane correlation functions by the limit
		\begin{equation*}
			\begin{split}
				&\ps{\L_{-1}V_\beta(t)\prod_{k=1}^NV_{\alpha_k}(z_k)\prod_{l=1}^MV_{\beta_l}(s_l)}\coloneqq\\
				&\lim\limits_{\delta,\eps,\rho\to0}\ps{\L_{-1}V_\beta(t)\prod_{k=1}^NV_{\alpha_k}(z_k)\prod_{l=1}^MV_{\beta_l}(s_l)}_{\delta,\eps,\rho}-\sum_{i=1}^2\tilde{\mathfrak{L}}^i_{-1,\delta,\eps,\rho}(\bm\alpha).
			\end{split}
		\end{equation*}
		Under the same assumptions, the $\Wb_{-1}V_\beta$ descendent field  is defined by setting
		\begin{equation*}
			\begin{split}
				&\ps{\Wb_{-1}V_\beta(t)\prod_{k=1}^NV_{\alpha_k}(z_k)\prod_{l=1}^MV_{\beta_l}(s_l)}\coloneqq\\
				&\lim\limits_{\delta,\eps,\rho\to0}\ps{\Wb_{-1}V_\beta(t)\prod_{k=1}^NV_{\alpha_k}(z_k)\prod_{l=1}^MV_{\beta_l}(s_l)}_{\delta,\eps,\rho}-\sum_{i=1}^2\tilde{\mathfrak{W}}^i_{-1,\delta,\eps,\rho}(\bm\alpha).
			\end{split}
		\end{equation*}
	\end{defi}
	Thanks to Lemma~\ref{lemma:desc1} we see that these limits are well-defined and analytic in a complex neighbourhood of $\mc A_{N,M+1}$.
    \begin{remark}
        Going along the proof of Lemma~\ref{lemma:desc1} we see that $\L_{-1}V_\beta$ is actually such that
        \begin{equation}\label{eq:expr_L1}
            \begin{split}
                &\ps{\L_{-1}V_\beta(t)\V}_{\delta,\eps}=\sum_{k=1}^N\frac{\ps{\beta,\alpha_k}}{2(z_k-t)}\ps{V_\beta(t)\V}_{\delta,\eps}-\Ir\left[\frac{\ps{\beta,\gamma e_i}}{2(x-t)}\ps{V_{\gamma e_i}(x)V_\beta(t)\V}_{\delta,\eps}\right]\\
                &+\Is\left[\left(\partial_x+\frac{\ps{\alpha_k,\gamma e_i}}{2(x-z_k)}\right)\ps{V_{\gamma e_i}(x)V_\beta(t)\V}_{\delta,\eps}\right]\\
                &-\Is\times\Ir\left[\frac{\ps{\gamma e_i,\gamma e_j}}{2(x-y)}\ps{V_{\gamma e_i}(x)V_{\gamma e_j}(y)V_\beta(t)\V}_{\delta,\eps}\right].
            \end{split}
        \end{equation}
    \end{remark}
	
	\subsubsection{Connection with derivatives of the correlation functions}
	The insertion of Virasoro descendants within correlation functions can be explicitly translated into derivatives of this correlation function. The following statement shows this connection for the $\L_{-1}$ descendant:
	\begin{proposition}\label{prop:L1_der}
		Take $t\in\R$ and $(\beta,\bm\alpha)\in\mc A_{N,M+1}$. Then in the sense of weak derivatives 
		\begin{equation*}
			\begin{split}
				&\ps{\L_{-1}V_\beta(t)\prod_{k=1}^NV_{\alpha_k}(z_k)\prod_{l=1}^MV_{\beta_l}(s_l)}=\partial_t \ps{V_\beta(t)\prod_{k=1}^NV_{\alpha_k}(z_k)\prod_{l=1}^MV_{\beta_l}(s_l)}.
			\end{split}
		\end{equation*}
	\end{proposition}
	\begin{proof}
		Note that for positive $\eps,\delta,\rho$, $\ps{V_\beta(t)\prod_{k=1}^NV_{\alpha_k}(z_k)\prod_{l=1}^MV_{\beta_l}(s_l)}_{\delta,\eps,\rho}$ is differentiable with respect to $t$ with derivative given by $\ps{\L_{-1}V_\beta(t)\prod_{k=1}^NV_{\alpha_k}(z_k)\prod_{l=1}^MV_{\beta_l}(s_l)}_{\delta,\eps,\rho}$. In particular if we take $f$ to be any smooth, compactly supported function over $\R$ then
		\begin{align*}
			&\int_{\R}f(t)\ps{\L_{-1}V_\beta(t)\prod_{k=1}^NV_{\alpha_k}(z_k)\prod_{l=1}^MV_{\beta_l}(s_l)}_{\delta,\eps,\rho}dt\\
			&=\int_{\R}\left(-\partial_t f(t)\right)\ps{V_\beta(t)\prod_{k=1}^NV_{\alpha_k}(z_k)\prod_{l=1}^MV_{\beta_l}(s_l)}_{\delta,\eps,\rho}dt.
		\end{align*}
		
		Let us first assume that $\ps{\beta,e_i}<0$ for $i=1,2$: in that case $\lim\limits_{\delta,\eps,\rho\to0}\tilde{\mathfrak{L}}^i_{-1,\delta,\eps,\rho}(\bm\alpha)=0$. To see why we rely on the explicit expression provided by Lemma~\ref{lemma:desc1}:
		\begin{align*}
			&\tilde{\mathfrak{L}}^i_{-1,\delta,\eps}(\bm\alpha)= \mu_{L,i}\ps{V_{\gamma e_i}(t-\eps)V_\beta(t)\V}_{\delta,\eps}-\mu_{R,i}\ps{V_{\gamma e_i}(t+\eps)V_\beta(t)\V}_{\delta,\eps}.
		\end{align*}
		Under the assumption that $\ps{\beta,e_i}<0$ for all $i$, as $\delta,\eps\to0$ we have using the fusion asymptotics (Lemma~\ref{lemma:fusion}) that $\ps{V_{\gamma e_i}(t\pm\eps)V_\beta(t)\V}_{\delta,\eps}\to0$, so that the above expression vanishes in the limit. This allows to conclude that if $\ps{\beta,e_i}<0$ for $i=1,2$ then
		\begin{align*}
			&\int_{\R}f(t)\ps{\L_{-1}V_\beta(t)\prod_{k=1}^NV_{\alpha_k}(z_k)\prod_{l=1}^MV_{\beta_l}(s_l)}dt=\int_{\R}\left(-\partial_t f(t)\right)\ps{V_\beta(t)\prod_{k=1}^NV_{\alpha_k}(z_k)\prod_{l=1}^MV_{\beta_l}(s_l)}dt.
		\end{align*}
		
		To extend this equality to the whole range of values of $(\beta,\bm\alpha)\in\mc A_{N,M+1}$ we rely on the fact that both the left and right-hand sides in this equality depend analytically on $\beta$. 
	\end{proof}

    The insertion of a $\L_{-1}$ descendant thus acts as a differential operator on the correlation functions. A major difference between Virasoro and $W$ currents is that this is not the case, for generic values of $\beta$, for the $\Wb_{-1}$ descendant. As we will see in~\cite{CH_sym2} (see also~\cite[Proposition 5.1]{Toda_OPEWV}) in order to be able to express the insertion of a $\Wb_{-1}$ as a differential operator we need to specialize to special values of the weight $\beta$, in which case we end up with \textit{singular vectors}.

	\subsection{Virasoro descendants and conformal Ward identities}
	Having defined the descendants at the first level we now turn to the definition of the Virasoro descendants at level higher than one. As we will see the proof is slightly more involved compared to that previously disclosed. As a consequence we may start using the shortcut $\Psi_{i_1,\cdots,i_l}(x_1,\cdots,x_l)$ defined in Equation~\eqref{eq:def_Psi}.
	
	\subsubsection{Virasoro descendants}
	We first introduce the Virasoro descendants $\L_{-n}V_\beta$. For this purpose and like before we start by introducing remainder terms as follows.
	\begin{lemma}\label{lemma:descn}
		For any positive integer $n$ and $i$ in $1,2$ set 
		\begin{equation*}
			\begin{split}
				\mathfrak{L}^i_{-n,\delta,\eps,\rho}(\bm\alpha)\coloneqq \Is\left[\partial_x\left(\frac{1}{(x-t)^{n-1}} \ps{V_{\gamma e_i}(x)V_\beta(t)\V}_{\delta,\eps,\rho} \right)\right].
			\end{split}
		\end{equation*}
		Then the following quantity is $(P)$-class:
		\begin{equation*}
			\begin{split}
				&\ps{\L_{-n}V_\beta(t)\prod_{k=1}^NV_{\alpha_k}(z_k)\prod_{l=1}^MV_{\beta_l}(s_l)}_{\delta,\eps,\rho}-\sum_{i=1}^2\mathfrak{L}^i_{-n,\delta,\eps,\rho}(\bm\alpha).
			\end{split}
		\end{equation*}
	\end{lemma}
	\begin{proof}
		The arguments are similar to the ones developed in the previous section. The starting point are the basic identities
		\begin{equation*}
			\L_{-n}^\beta\left(\alpha \ln\frac1{\norm{\cdot-t}}\right)=\frac{(n-1)\Delta_{\alpha}+\frac{\ps{\beta,\gamma e_i}}{2}}{(\cdot-t)^2}\qt{and} \L_{-n}^\beta\left(\gamma e_i\ln\frac1{\norm{\cdot-t}}\right)=\frac{n-1+\frac{\ps{\beta,\gamma e_i}}{2}}{(\cdot-t)^2}\cdot
		\end{equation*}
		Thanks to them we can write that
		\begin{align*}
			&\ps{\L_{-n}V_\beta(t)\prod_{k=1}^NV_{\alpha_k}(z_k)\prod_{l=1}^MV_{\beta_l}(s_l)}_{\delta,\eps}=\\
			&\left[\sum_{k=1}^{2N+M} \frac{(n-1)\Delta_{\alpha_k}}{2(z_k-t)^{n}}-\sum_{k=1}^{2N+M} \left(\frac{\ps{\beta,\alpha_k}}{2(z_k-t)^{n}}+ \sum_{p=1}^{n-1}\sum_{l\neq k}\frac{\ps{\alpha_k,\alpha_l}}{4(z_k-t)^p(z_l-t)^{n-p}}\right)\right]\ps{V_\beta(t)\V}_{\delta,\eps} \\
			&- \It \left[\left(\frac{n-1+\frac{\ps{\beta,\gamma e_i}}2}{(x-t)^{n}}-\sum_{p=1}^{n-1}\sum_{k=1}^{2N+M}\frac{\ps{\alpha_k,\gamma e_i}}{2(x-t)^p(z_k-t)^{n-p}}\right)\Psi_i(x)\right]\\
			&+ \sum_{i=1}^2\mu_{B,i}\int_{\Heps} \sum_{p=1}^{n-1}\frac{\ps{\gamma e_i,\gamma e_i}}{2(x-t)^p(\bar x-t)^{n-p}}\Psi_i(x)d^2x-\It^2\left[\sum_{p=1}^{n-1}\frac{\ps{\gamma e_i,\gamma e_j}}{4(x-t)^p(y-t)^{n-p}}\Psi_{i,j}(x,y)\right].
		\end{align*}
		
		Let us start with the one-fold integral. Since the integrals over $\Heps^c$ and $\Reps^c$ are $(P)$-class  we can focus on the singular part of these integrals.
		For this we rely on the fact that 
		\begin{equation}\label{eq:sym_id}
			\sum_{p=1}^{n-1}\frac{1}{(x-t)^p(y-t)^{n-p}}=\frac{1}{x-y}\left(\frac{1}{(y-t)^{n-1}}-\frac{1}{(x-t)^{n-1}}\right),
		\end{equation}
		so that the singular one-fold integrals can be rewritten as
		\begin{align*}
			&- \sum_{k=1}^{2N+M}\frac1{(z_k-t)^{n-1}}\Is \left[\frac{\ps{\alpha_k,\gamma e_i}}{2(z_k-x)} \Psi_i(x)\right]\\
			&- \Is \left[\left(\frac{n-1+\frac{\ps{\beta,\gamma e_i}}2}{(x-t)^n}+\sum_{k=1}^{2N+M}\frac{\ps{\alpha_k,\gamma e_i}}{2(x-t)^{n-1}(x-z_k)}\right) \Psi_i(x)\right]\\
			&- \sum_{i=1}^2\mu_{B,i}\int_{\Heps} \left(\frac{\ps{\gamma e_i,\gamma e_i}}{2(x-t)^{n-1)}(x-\bar x)}+\frac{\ps{\gamma e_i,\gamma e_i}}{2(\bar x-t)^{n-1}(\bar x-x)}\right)\Psi_i(x)d^2x.
		\end{align*}
		The first term is $(P)$-class; the second one can be dealt with like before by writing it as a total derivative in that
		\begin{align} \label{eq:tot deriv 1}
			&\partial_x\left(\frac{1}{(t-x)^{n-1}}\Psi_i(x)\right)\\
			&=\left(\frac{n-1+\frac{\ps{\beta,\gamma e_i}}2}{(t-x)^n}+\sum_{k=1}^{2N+M}\frac{\ps{\alpha_k,\gamma e_i}}{2(t-x)^{n-1}(z_k-x)}+\frac{\ps{\gamma e_i,\gamma e_i}}{2(t-x)^{n-1)}(\bar x-x)}\right) \Psi_i(x)\nonumber\\
			&-\It\left(\frac{\ps{\gamma e_i,\gamma e_j}}{2(t-x)^{n-1}(y-x)}\Psi_i(x)\right)\nonumber.
		\end{align}
		
		Therefore the only remaining terms to be treated are given by
		\begin{align*}
			&-\Is\times\It \left(\frac{\ps{\gamma e_i,\gamma e_j}}{2(x-t)^{n-1}(y-x)} \Psi_i(x)\right)\\
			&+\It^2\left(\sum_{i=1}^{n-1}\frac{\ps{\gamma e_i,\gamma e_j}}{4(x-t)^i(y-t)^{n-i}}\Psi_{i,j}(x,y)\right).
		\end{align*}
		We can rewrite the second integral using the identity~\eqref{eq:sym_id} together with symmetry in the $x,y$ variable. By doing so we see that the integral over $\Is\times\It$ vanishes, so that only the following integral remains: 
		\begin{align*}
			&\Ir\times\It \left(\frac{\ps{\gamma e_i,\gamma e_j}}{2(x-t)^{n-1}(y-x)} \Psi_i(x)\right).
		\end{align*}
		Over $\Ir\times\Ir$ we use again the identity~\eqref{eq:sym_id} to write the integrand as 
        $$\sum_{p=1}^{n-1}\frac{\ps{\gamma e_i,\gamma e_j}}{2(x-t)^{n-p}(y-t)^p} \Psi_{i,j}(x,y)$$
        so that this term is $(P)$-class. As for the integral over $\Ir\times\Is$ it is seen to be $(P)$-class thanks to the fusion asymptotics of Lemma~\ref{lemma:fusion_integrability}.
	\end{proof}
	
	\begin{defi}
		We then define the $L_{-n}V_\beta$ descendant within correlation functions via the limits
		\begin{equation*}
			\begin{split}
				&\ps{\L_{-n}V_\beta(t)\prod_{k=1}^NV_{\alpha_k}(z_k)\prod_{l=1}^MV_{\beta_l}(s_l)}\coloneqq\\
				&\lim\limits_{\delta,\eps,\rho\to0}\ps{\L_{-n}V_\beta(t)\prod_{k=1}^NV_{\alpha_k}(z_k)\prod_{l=1}^MV_{\beta_l}(s_l)}_{\delta,\eps,\rho}-\sum_{i=1}^2\tilde{\mathfrak{L}}^i_{-n,\delta,\eps,\rho}(\bm\alpha).
			\end{split}
		\end{equation*}
	\end{defi}
	Like before we can provide an explicit expression for the remainder terms defined above:
	\begin{equation*}
		\begin{split}
			&\tilde{\mathfrak{L}}^i_{-n,\delta,\eps,\rho}(\bm\alpha)=\left(\mu_{L,i}\frac{\ps{V_{\gamma e_i}(t-\eps)V_\beta(t)\V}_{\delta,\eps}}{(-\eps)^{n-1}}-\mu_{R,i}\frac{\ps{V_{\gamma e_i}(t+\eps)V_\beta(t)\V}_{\delta,\eps}}{(\eps)^{n-1}}\right)\\
			&-\mu_{B,i}\int_{(t-r,t+r)}\ps{V_{\gamma e_i}(x+i\delta)V_\beta(t)\V}_{\delta,\eps}\Im\left(\frac{1}{(x+i\delta-t)^{n-1}}\right)dx.
		\end{split}
	\end{equation*}
	And like before these remainder terms and the one provided in Lemma~\ref{lemma:descn} differ by quantities that are $(P)$-class.	As such the limit defining $\ps{\L_{-n}V_\beta(t)\prod_{k=1}^NV_{\alpha_k}(z_k)\prod_{l=1}^MV_{\beta_l}(s_l)}$ is well-defined and analytic in a complex neighbourhood of $\mc A_{N,M+1}$.
    Likewise the $\L_{-n}V_\beta$ descendant when inserted within correlation functions is seen to be given at the regularized level by
    \begin{equation}\label{eq:expr_Ln}
        \begin{split}
			&\ps{\L_{-n}V_\beta(t)\V}_{\delta,\eps}=A_{\delta,\eps}+B_{\delta,\eps},\qt{where}\\
            &A_{\delta,\eps}\coloneqq\left(\sum_{k=1}^{2N+M}\frac{\ps{(n-1)Q+\beta,\alpha_k}}{2(z_k-t)^n}-\sum_{k,l=1}^{2N+M}\sum_{p=1}^{n-1}\frac{\ps{\alpha_k,\alpha_l}}{4(z_k-t)^p(z_l-t)^{n-p}}\right)\ps{V_\beta(t)\V}_{\delta,\eps}\\
            &-\Ir\left[\left(\frac{n-1+\frac{\ps{\beta,\gamma e_i}}2}{(x-t)^n}-\sum_{k=1}^{2N+M}\sum_{p=1}^{n-1}\frac{\ps{\gamma e_i,\alpha_k}}{2(x-t)^p(z_k-t)^{n-p}}\right)\Psi_i(x)\right]\\
            &-\Is\left[\sum_{k=1}^{2N+M}\frac{\ps{\gamma e_i,\alpha_k}}{2(z_k-t)^{n-1}(z_k-x)}\Psi_i(x)\right]\\
            &+\Ir\times\Ir\left[-\sum_{p=1}^{n-1}\frac{\ps{\gamma e_i,\gamma e_j}}{4(x-t)^{p}(y-t)^{n-p}}\Psi_{i,j}(x,y)\right]\\
            &+\Is\times\Ir\left[\frac{\ps{\gamma e_i,\gamma e_j}}{2(y-t)^{n-1}(y-x)}\Psi_{i,j}(x,y)\right]\\
            &\text{while}\quad B_{\delta,\eps}\coloneqq\Is\left[\partial_x\left(\frac1{(x-t)^{n-1}}\Psi_i(x)\right)\right].
        \end{split}
     \end{equation}

	\subsubsection{Conformal Ward identities}
	We have seen previously that the $\L_{-1}$ descendant allowed to define the (weak) derivatives of the correlation functions. As we now show the other descendants $\L_{-n}$, $n\geq 2$ play an analogous role based on the conformal Ward identities.
	\begin{theorem}\label{thm:ward_vir}
		Take $t\in\R$ and $(\beta,\bm\alpha)\in\mc A_{N,M+1}$. Then for any $n\geq2$, in the sense of weak derivatives 
		\begin{equation*}
			\begin{split}
				&\ps{\L_{-n}V_\beta(t)\prod_{k=1}^NV_{\alpha_k}(z_k)\prod_{l=1}^MV_{\beta_l}(s_l)}\\
                &=\left(\sum_{k=1}^{2N+M}\frac{-\partial_{z_k}}{(z_k-t)^{n-1}}+\frac{(n-1)\Delta_{\alpha_k}}{(z_k-t)^n}\right) \ps{V_\beta(t)\prod_{k=1}^NV_{\alpha_k}(z_k)\prod_{l=1}^MV_{\beta_l}(s_l)}.
			\end{split}
		\end{equation*}
	\end{theorem}
	\begin{proof}
		As explained in Proposition~\ref{prop:L1_der} the weak derivatives are defined using the $\L_{-1}$ descendant fields, themselves defined using a regularization procedure based on Lemma~\ref{lemma:desc1}. As such our goal is to show that
		\begin{equation}\label{eq:to_prove_Ln}
			\begin{split}
				\lim\limits_{\delta,\eps,\rho\to 0}&\left(\Lc_{-n}-\sum_{k=1}^{2N+M}\frac{-\partial_{z_k}}{(z_k-t)^{n-1}}+\frac{(n-1)\Delta_{\alpha_k}}{(z_k-t)^n}\right) \ps{V_\beta(t)\prod_{k=1}^NV_{\alpha_k}(z_k)\prod_{l=1}^MV_{\beta_l}(s_l)}_{\delta,\eps,\rho}\\
				&\quad-\sum_{i=1}^2\mathfrak R^i(\delta,\eps,\rho;\bm\alpha)=0,\qt{with}
			\end{split}
		\end{equation}
		\begin{align*}
			&\Lc_{-n}\ps{V_\beta(t)\prod_{k=1}^NV_{\alpha_k}(z_k)\prod_{l=1}^MV_{\beta_l}(s_l)}_{\delta,\eps,\rho}\coloneqq\ps{\L_{-n}V_\beta(t)\prod_{k=1}^NV_{\alpha_k}(z_k)\prod_{l=1}^MV_{\beta_l}(s_l)}_{\delta,\eps,\rho}\qt{and}\\
			&\mathfrak R^i(\delta,\eps,\rho;\bm\alpha)=\tilde{\mathfrak{L}}^i_{-n,\delta,\eps,\rho}(\bm\alpha)+\sum_{l=1}^M\frac{\tilde{\mathfrak{L}}^i_{-1,\delta,\eps,\rho}(\bm\alpha;\beta_l)}{(s_l-t)^{n-1}}.
		\end{align*}
		In the above $\tilde{\mathfrak{L}}^i_{-1,\delta,\eps,\rho}(\bm\alpha;\beta_l)$ is the remainder term from Lemma~\ref{lemma:desc1} but associated to the boundary descendant field $\L_{-1}V_{\beta_l}(s_l)$. 
		
		To see why this is true we combine the exact expressions provided in the proofs of Lemmas~\ref{lemma:desc1} and~\ref{lemma:descn} to get that the first line in Equation~\eqref{eq:to_prove_Ln} is given by 
		\begin{align*}
			&- \It \left[\left(\frac{n-1+\frac{\ps{\beta,\gamma e_i}}2}{(x-t)^{n}}+\sum_{k=1}^{2N+M}\frac{\ps{\alpha_k,\gamma e_i}}{2(x-t)^{n-1}(z_k-x)}\right) \Psi_i(x)\right]\\
			&- \sum_{i=1}^2\mu_{B,i}\int_{\Heps} \left(\frac{\ps{\gamma e_i,\gamma e_i}}{2(x-t)^{n-1)}(x-\bar x)}+\frac{\ps{\gamma e_i,\gamma e_i}}{2(\bar x-t)^{n-1}(\bar x-x)}\right)\Psi_i(x)d^2x\\
			&+\It^2\left(\sum_{i=1}^{n-1}\frac{\ps{\gamma e_i,\gamma e_j}}{2(x-t)^{n-1}(x-y)}\Psi_{i,j}(x,y)\right).
		\end{align*}
		Along the same lines as before the latter is a total derivative. We get that the first line in Equation~\eqref{eq:to_prove_Ln} is equal to
		\begin{align*}
			& \It \left[\partial_x\left(\frac{1}{(x-t)^{n-1}}\Psi_i(x)\right) \right].
		\end{align*}
		After integration by parts (that is application of Stoke's formula over the domains $\Heps$ and $\Reps$) we recover the desired remainder terms provided that for any $1\leq l\leq M$,
		\begin{align*}
			\lim\limits_{\delta,\eps,\rho\to0} \left(\frac{1}{(s_l+\eps-t)^{n-1}}-\frac{1}{(s_l-t)^{n-1}}\right)\Psi_i(s_l+\eps)=0,
		\end{align*}
		and likewise for any $1\leq k\leq N$:
		\begin{align*}
			\lim\limits_{\delta,\eps,\rho\to0} \oint_{\partial B(z_k,\eps)}\Psi_i(\xi)d\xi=0,
		\end{align*}
		This follows from the fusion asymptotics of Lemma~\ref{lemma:fusion}.
	\end{proof}

    \subsection{$W$-descendants and higher-spin Ward identities}
    Having shown the validity of conformal Ward identities, we now turn to the higher-spin symmetry enjoyed by Toda CFTs and which are encoded using the current $\Wb$. To this end we first define the $W$-descendants at the order two and then proceed to the general case of $\Wb_{-n}$ with $n\geq 3$. We will show that the insertion of a $\Wb_{-n}$ descendant, $n\geq 3$, within correlation functions has an effect similar to that of the Virasoro descendant $\L_{-n}$ but that involves this time the $\Wb$ descendants corresponding to the other insertions.
    
	\subsubsection{Definition of the $W$-descendants}
	In order to define the $W$-descendants we proceed in the same fashion that we did for the Virasoro descendants. And to start with we define the descendant at the level two.
	\begin{lemma}\label{lemma:desc_w2}
		Let us set for $i=1,2$: 
		\begin{equation} \label{eq:rem_W2}
            \begin{split}
			\mathfrak{W}^i_{-2,\delta,\eps,\rho}(\bm\alpha)&\coloneqq\Is\left[\partial_x\left(\left(\frac{2h_2(e_i)\omega_{\hat i}(\beta)}{x-t}+\sum_{k=1}^{2N+M}\frac{2h_2(e_i)\omega_{\hat i}(\alpha_k)}{x-z_k}\right)\Psi_i(x)\right)\right]\nonumber\\
			&-\Is\times \It\left[\delta_{i\neq j}\partial_x\left(\frac{2\gamma h_2(e_i)}{x-y}\Psi_{i,j}(x,y)\right)\right]
        \end{split}
		\end{equation}
        where recall the notation $\hat i=3-i$.	Then the following quantity is $(P)$-class:
		\begin{equation*}
			\begin{split}
				&\ps{\Wb_{-2}V_\beta(t)\prod_{k=1}^NV_{\alpha_k}(z_k)\prod_{l=1}^MV_{\beta_l}(s_l)}_{\delta,\eps,\rho}-\sum_{i=1}^2\mathfrak{W}^i_{-2,\delta,\eps,\rho}(\bm\alpha).
			\end{split}
		\end{equation*}
		We define the $\Wb_{-2}$ descendant by setting
		\begin{equation*}
			\begin{split}
				&\ps{\Wb_{-2}V_\beta(t)\prod_{k=1}^NV_{\alpha_k}(z_k)\prod_{l=1}^MV_{\beta_l}(s_l)}\coloneqq\\
				&\lim\limits_{\delta,\eps,\rho\to0}\ps{\Wb_{-2}V_\beta(t)\prod_{k=1}^NV_{\alpha_k}(z_k)\prod_{l=1}^MV_{\beta_l}(s_l)}_{\delta,\eps,\rho}-\sum_{i=1}^2\tilde{\mathfrak{W}}^i_{-2,\delta,\eps,\rho}(\bm\alpha).
			\end{split}
		\end{equation*}
	\end{lemma}
	\begin{proof}
		The proof relies on the very same arguments as the one of Lemma~\ref{lemma:descn}, the main difference lying in the combinatorial identities underlying the definition of the $W$-descendants. To be more specific we rely on the fact that (shown by explicit computations) for $i=1,2$:
		\begin{align*}
			&\Wb_{-2}^\beta\left(\gamma e_i \ln\frac{1}{\norm{\cdot-t}}\right)=2h_2(e_i)w_{\hat i}(\beta)\frac{1+\frac{\ps{\beta,\gamma e_i}}{2}}{(\cdot-t)^2};\\
            &C^\sigma(\beta,\gamma e_i,\alpha_k)=-h_2(e_i)\left(\omega_{\hat i}(\alpha_k)\ps{\beta,\gamma e_i}+\omega_{\hat i}(\beta)\ps{\alpha_k,\gamma e_i}\right);\\
            &C^\sigma(\beta,\gamma e_i,\gamma e_j)= -\gamma h_2(e_i)\ps{\beta,\gamma e_i}\delta_{i\neq j}-\ps{\gamma e_i,\gamma e_j} h_2(e_i)\omega_{\hat{i}}(\beta). 
		\end{align*}
        As a consequence we can write that the integrals appearing at the regularized level are
        \begin{align*}
            &-\Is\left[\left(\frac{2 h_2(e_i)\omega_{\hat i}(\beta)\left(1+\ps{\beta,\gamma e_i}/2\right)}{(t-x)^2}-\sum_{k=1}^{2N+M}\frac{h_2(e_i)\left(\omega_{\hat i}(\alpha_k)\ps{\beta,\gamma e_i}+\omega_{\hat i}(\beta)\ps{\alpha_k,\gamma e_i}\right)}{(t-x)(t-z_k)}\right)\Psi_i(x)\right]\\
            &+\It^2\left[\delta_{i=j}\frac{h_2(e_i)w_{\hat i}(\beta)\times -\ps{\gamma e_i,\gamma e_i}}{2(t-x)(t-y)}\Psi_{i,j}(x,y)\right]+\It^2\left[\delta_{i\neq j}\frac{-2\gamma^2h_2(\beta)}{(t-x)(t-y)}\Psi_{i,j}(x,y)\right].
        \end{align*}
		Proceeding along the very same lines as in Lemma~\ref{lemma:descn} using symmetrization and integration by parts, we can rewrite the latter (up to convergent terms) as
        \begin{align*}
            &-\Is\left[\partial_x\left(\left(\frac{2 h_2(e_i)\omega_{\hat i}(\beta)}{t-x}-\sum_{k=1}^{2N+M}\frac{2h_2(e_i)\omega_{\hat i}(\alpha_k)}{x-z_k}\right)\Psi_i(x)\right)\right]\\
            &+\Is\times\It\left[\frac{h_2(e_i)\omega_{j}(\beta)\ps{\gamma e_i,\gamma e_j}}{(t-x)(x-y)}\Psi_{i,j}(x,y)\right]\\
            &+\It^2\left[\delta_{i=j}\frac{h_2(e_i)w_{\hat i}(\beta)\times \ps{\gamma e_i,\gamma e_i}}{(t-x)(y-x)}\Psi_{i,j}(x,y)\right]+\It^2\left[\delta_{i\neq j}\frac{-2\gamma^2h_2(\beta)}{(t-x)(t-y)}\Psi_{i,j}(x,y)\right].
        \end{align*}
        The only remaining two-fold integrals are thus given by
        \begin{align*}
            &\Is\times\It\left[\left(\delta_{i\neq j}\frac{-\gamma^2h_2(e_i)\omega_{j}(\beta)}{(t-x)(x-y)}+\frac{-2\gamma^2h_2(\beta)}{(t-x)(t-y)}\right)\Psi_{i,j}(x,y)\right].
        \end{align*}
        Since for $i\neq j$, $h_2(e_i)\omega_j(\beta)+2h_2(\beta)=h_2(e_i)\ps{e_i,\beta}$ the latter is actually equal (again up to convergent terms) to, as expected,
        \begin{align*}
            &\Is\times\It\left[\delta_{i\neq j}\partial_x\left(\frac{2\gamma h_2(e_i)}{y-x}\Psi_{i,j}(x,y)\right)\right]
        \end{align*}
	\end{proof}
	
	We now turn to the $W$-descendants at level higher than two.
	\begin{lemma}\label{lemma:desc_wn}
		For any positive integer $n$ and $i$ in $1,2$, set 
        \begin{align} \label{eq: true remainder wn}
			\mathfrak{W}^i_{-n,\delta,\eps,\rho}(\bm\alpha)&\coloneqq\Is\left[\partial_x\left(\left(\frac{h_2(e_i)((n-2)q+2\omega_{\hat{i}}(\beta))}{(x-t)^{n-1}}+\sum_{k=1}^{2N+M}\frac{2h_2(e_i)\omega_{\hat{i}}(\alpha_k)}{(x-t)^{n-2}(x-z_k)}\right)\Psi_{i}(x)\right)\right]\nonumber\\
			&+\Is\times \It\left[\partial_x\left(\frac{2\gamma h_2(e_i)\delta_{i\neq j}}{(x-t)^{n-2}(y-x)}\Psi_{i,j}(x,y)\right)\right].\nonumber
		\end{align}
		Then the following quantity is $(P)$-class:
		\begin{equation*}
			\begin{split}
				&\ps{\Wb_{-n}V_\beta(t)\prod_{k=1}^NV_{\alpha_k}(z_k)\prod_{l=1}^MV_{\beta_l}(s_l)}_{\delta,\eps,\rho}-\sum_{i=1}^2\mathfrak{W}^i_{-n,\delta,\eps,\rho}(\bm\alpha).
			\end{split}
		\end{equation*}
	\end{lemma}
	\begin{proof}
		The result follows from the proof of Theorem~\ref{thm:ward_Wn} in the next subsection.
	\end{proof}
	
	\begin{defi} \label{defi: desc wn}
		Take $t\in\R$ and $(\beta,\bm\alpha)\in\mc A_{N,M+1}$. For $n\geq 3$ we define the descendant field $\Wb_{-n}V_\beta$ within half-plane correlation functions by the limit
		\begin{equation*}
			\begin{split}
				&\ps{\Wb_{-n}V_\beta(t)\prod_{k=1}^NV_{\alpha_k}(z_k)\prod_{l=1}^MV_{\beta_l}(s_l)}\coloneqq\\
				&\lim\limits_{\delta,\eps,\rho\to0}\ps{\Wb_{-n}V_\beta(t)\prod_{k=1}^NV_{\alpha_k}(z_k)\prod_{l=1}^MV_{\beta_l}(s_l)}_{\delta,\eps,\rho}-\sum_{i=1}^2\tilde{\mathfrak{W}}^i_{-n,\delta,\eps,\rho}(\bm\alpha).
			\end{split}
		\end{equation*}
	\end{defi}
	Like before, these limits are well-defined and analytic in a complex neighborhood of $\mc A_{N,M+1}$. The explicit expression for the remainder is 
        \begin{equation} \label{eq: remainder wn tilda}
            \begin{split}
            &\tilde{\mathfrak{W}}^i_{-n,\delta,\eps,\rho}(\bm\alpha) \\
            &=h_2(e_i)((n-2)q+2\omega_{\hat{i}}(\beta))\left(\mu_{i,L}\frac{\Psi_i(t-\eps)}{(-\eps)^{n-1}}-\mu_{i,R}\frac{\Psi_i(t-\eps)}{\eps^{n-1}}\right)\\
            &+\sum_{k} 2h_2(e_i)\omega_{\hat{i}}(\alpha_k)\left(\mu_{i,L}\frac{\Psi_i(t-\eps)}{(-\eps)^{n-2}(t-\eps-z_k)}-\mu_{i,R}\frac{\Psi_i(t+\eps)}{\eps^{n-2}(t+\eps-z_k)}\right)\\
            &-\mu_{B,i}\int_{(t-r,t+r)} \Im \left( \frac{h_2(e_i)((n-2)q+2\omega_{\hat{i}}(\beta))}{(x+i\delta-t)^{n-1}} + \sum_k \frac{2h_2(e_i)\omega_{\hat{i}}(\alpha_k)}{(x+i\delta-t)^{n-2}(x+i\delta-z_k)}\right)\Psi_i(x+i\delta)dx\\
            &+\It\left[2\gamma h_2(e_i)\delta_{i\neq j}\left(\mu_{i,L}\frac{\Psi_{i,j}(t-\eps,y)}{(s_l-\eps-t)^{n-2}(y-s_l+\eps)}-\mu_{i,R}\frac{\Psi_{i,j}(t+\eps,y)}{(s_l+\eps-t)^{n-2}(y-s_l-\eps)}\right)\right.\\
            &\left.+\mu_{B,i}\int_{(t-r,t+r)} \left(\frac{2\gamma h_2(e_i)\delta_{i\neq j}}{(x+i\delta-t)^{n-2}(y-x-i\delta)}-\frac{2\gamma h_2(e_i)\delta_{i\neq j}}{(x-i\delta-t)^{n-2}(y-x+i\delta)}\right)\Psi_{i,j}(x+i\delta,y)dx\right]
            \end{split}
        \end{equation}
	
	\subsubsection{Ward identities for the higher-spin current}
	In the same fashion as the Virasoro descendants, the insertion of $W$-descendants at order higher than three within correlation functions gives rise to Ward identities but this time associated to higher-spin symmetry. They take the following form:
	\begin{theorem}\label{thm:ward_Wn}
		Take $t\in\R$ and $(\beta,\bm\alpha)\in\mc A_{N,M+1}$. Then for any $n\geq3$:
		\begin{equation} \label{eq:ward_Wn}
			\begin{split}
                &\ps{\Wb_{-n}V_\beta(t)\prod_{k=1}^NV_{\alpha_k}(z_k)\prod_{l=1}^MV_{\beta_l}(s_l)}= \\
                  &\left(\sum_{k=1}^{2N+M}\frac{-\Wc_{-2}^{(k)}}{(z_k-t)^{n-2}}+\frac{(n-2)\Wc_{-1}^{(k)}}{(z_k-t)^{n-1}}-\frac{(n-1)(n-2)w(\alpha_k)}{2(z_k-t)^n}\right) \ps{V_\beta(t)\prod_{k=1}^NV_{\alpha_k}(z_k)\prod_{l=1}^MV_{\beta_l}(s_l)}
			\end{split}
		\end{equation}
		where for $j=1,2$, $\Wc_{-j}^{(k)}\ps{V_\beta(t)\V}\coloneqq \ps{\Wb_{-j}V_{\alpha_k}(z_k)\prod_{l\neq k}V_{\alpha_l}(z_l)}$, and $\Wc_{-j}^{(k)} = \overline{\Wc}_{-j}^{(k-N)}$ for $k\in\{N+1,...,2N\}$.
	\end{theorem}
	\begin{proof}
		\emph{The free-field theory.} In order to make the exposition as clear as possible we start by considering the free-field theory, that is when $\mu_i = \mu_i^\partial = 0$. In this setting we have to assume that $\sum_k \alpha_k + \frac12\sum_l \beta_l - Q=0$ in order for the Gaussian integration by parts formula (Lemma \ref{lemma:GaussianIPP}) to be valid. To start with from the explicit definitions of the $W$-descendants and thanks to Lemma \ref{lemma:GaussianIPP} we see that the right-hand side of \eqref{eq:ward_Wn} is given by 
         \begin{align} \label{eq:ward rhs free field}
			&\left(\sum_{k=1}^{2N+M}-\frac{\Wc_{-2}^{(k)}}{(z_k-t)^{n-2}}+\frac{(n-2)\Wc_{-1}^{(k)}}{(z_k-t)^{n-1}}-\frac{(n-1)(n-2)w(\alpha_k)}{2(z_k-t)^n}\right) \ps{V_\beta(t)\V}_{\delta,\eps}^{FF} \\
            &= -\left[\sum_{k=1}^{2N+M}\left(  \frac{(n-1)(n-2)w(\alpha_k)-(n-2)(qB(\alpha_k,\beta)+2C(\alpha_k,\alpha_k,\beta))}{2(z_k-t)^n}\right.\right.\nonumber\\
            &\left.+\frac{q(B(\beta,\alpha_k)-B(\alpha_k,\beta))+2C(\beta-\alpha_k,\alpha_k,\beta)}{2(z_k-t)^n}\right)\nonumber\\
			&+\sum_{l\neq k} \left( \frac{-(n-2)(q B(\alpha_k,\alpha_l)+2C(\alpha_k,\alpha_k,\alpha_l))+2C(\alpha_k,\beta,\alpha_l)+2C(\alpha_k,\alpha_l,\beta)}{2(z_k-t)^{n-1}(z_k-z_l)} \right.\nonumber\\
            &+\left.\frac{q(B(\alpha_l,\alpha_k)-B(\alpha_k,\alpha_l))+2C(\alpha_l-\alpha_k,\alpha_k,\alpha_l)}{2(z_k-t)^{n-2}(z_k-z_l)^2} \right)\nonumber\\
            &+\left.\sum_{m\neq l,k} \frac{C(\alpha_k,\alpha_l,\alpha_m)}{(z_k-t)^{n-2}(z_k-z_l)(z_k-z_m)} \right] \ps{V_\beta(t)\V}_{\delta,\eps}^{FF} \nonumber
		\end{align}
        where the notation $\ps{\cdot}^{FF}$ refers to correlation functions with respect to the free field. 
        
        On the other hand, based on the defining formula~\eqref{eq:W_desc} and using again Lemma \ref{lemma:GaussianIPP} we see that the left-hand side of \eqref{eq:ward_Wn} is equal to :
		\begin{align} \label{eq:wn free field}
			&\ps{\Wb_{-n}V_\beta(t)\V}^{FF}_{\delta,\eps} \\
            &= -\left[\sum_{k=1}^{2N+M} \left( \frac{qB(\beta,\alpha_k)+2C(\beta,\beta,\alpha_k)-(n-1)(qB(\alpha_k,\beta)+2C(\alpha_k,\alpha_k,\beta))}{2(z_k-t)^n}\right.\right.\nonumber\\
            &\left.+\frac{(n-1)(n-2)(-q^2h_2(\alpha_k)+\frac{q}{2}B(\alpha_k,\alpha_k)+\frac13 C(\alpha_k,\alpha_k,\alpha_k))}{2(z_k-t)^n}\right)\nonumber\\
			&+ \sum_{l\neq k} \left( \sum_{p=1}^{n-1} \frac{q (p-1) B(\alpha_k,\alpha_l)-2C(\alpha_k,\alpha_l,\beta)+2(p-1)C(\alpha_k,\alpha_k,\alpha_l)}{2(z_k-t)^{p}(z_l-t)^{n-p}} \right.\nonumber\\
			& + \left.\left.\sum_{m\neq l,k} \sum_{p_1=1}^{n-2}\sum_{p_2=1}^{p_1}\frac{C(\alpha_k,\alpha_l,\alpha_m)}{3(z_k-t)^{p_2}(z_l-t)^{p_1-p_2+1}(z_m-t)^{n-p_1-1}} \right)\right] \ps{V_\beta(t)\V}_{\delta,\eps}^{FF}. \nonumber
		\end{align}
		Elementary calculations involving the explicit expressions for $B$ and $C$ show that the $(z_k-t)^{-n}$ terms in \eqref{eq:wn free field} and \eqref{eq:ward rhs free field} are actually equal.  In addition to the symmetrization identity from Equation~\eqref{eq:sym_id} we use the following additional symmetrization identities to transform the terms involving several insertions in Equation~\eqref{eq:wn free field}:
		\begin{equation} \label{eq: sym id 2}
			\sum_{p=1}^{n-1} \frac{p-1}{(t-x)^p(t-y)^{n-p}} = \frac{1}{(y-x)^2}\left(\frac{1}{(t-y)^{n-2}}-\frac{1}{(t-x)^{n-2}}\right)-\frac{n-2}{(y-x)(t-x)^{n-1}};
		\end{equation}
		\begin{align} \label{eq: sym id 3}
			\sum_{p_1=1}^{n-2}\sum_{p_2=1}^{p_1} &\frac{1}{(t-x)^{p_2}(t-y)^{p_1-p_2+1}(t-z)^{n-p_1-1}} = \frac{1}{(z-y)(z-x)(t-z)^{n-2}}\\ &+\frac{1}{(y-x)(y-z)(t-y)^{n-2}}+\frac{1}{(x-y)(x-z)(t-x)^{n-2}} \cdot \nonumber
		\end{align}
        Now we can use Equations~\eqref{eq:sym_id} and~\eqref{eq: sym id 2} to symmetrize the terms involving $k,l$ insertions in the left-hand side of the Ward identity~\eqref{eq:wn free field}. Doing so explicit computations show that this yields precisely the corresponding term in the right-hand side of Equation~\eqref{eq:ward rhs free field}. For the last line of~\eqref{eq:wn free field} we use the identity~\eqref{eq: sym id 3} together with symmetry between the $z_k$, $z_l$ and $z_m$ variables: we then see that this exactly matches the last line of Equation~\eqref{eq:ward rhs free field}. 
        Finally, we deduce that the Ward identity for the free-field theory holds true:
		\begin{equation*}
			\begin{split}
				&\ps{\Wb_{-n}V_\beta(t)\V}^{FF}\\
                &=\left(\sum_{k=1}^{2N+M}\frac{-\Wc_{-2}^{(k)}}{(z_k-t)^{n-2}}+\frac{(n-2)\Wc_{-1}^{(k)}}{(z_k-t)^{n-1}}-\frac{(n-1)(n-2)w(\alpha_k)}{2(z_k-t)^n}\right) \ps{V_\beta(t)\V}^{FF}
			\end{split}
		\end{equation*}
        after taking the limit as $\eps,\delta\to0$.

        \emph{The full theory.}
		The reasoning is the same in the presence of the GMC potential except that we have to work at the regularized level and keep track of the remainder terms that may occur. At such our goal is to show that
        \begin{equation} \label{eq:ward wn with remainder}
            \begin{split}
                &\lim_{\delta,\eps,\rho\to 0}\left(\Wc_{-n}-\left(\sum_{k=1}^{2N+M}\frac{-\Wc_{-2}^{(k)}}{(z_k-t)^{n-2}}+\frac{(n-2)\Wc_{-1}^{(k)}}{(z_k-t)^{n-1}}-\frac{(n-1)(n-2)w(\alpha_k)}{2(z_k-t)^n}\right)\right) \ps{V_\beta(t)\V}_{\delta,\eps,\rho}\\
                &-\sum_{i=1}^2 \mathfrak{R}^i(\delta,\eps,\rho;\bm\alpha) = 0,\qt{with}\\
                &\mathfrak{R}^i(\delta,\eps,\rho;\bm\alpha) := \tilde{\mathfrak{W}}^i_{-n,\delta,\eps,\rho}(\bm\alpha) - \sum_{l=1}^M \frac{-\tilde{\mathfrak{W}}^i_{-2,\delta,\eps,\rho}(\bm\alpha;\beta_l)}{(s_l-t)^{n-2}}+\frac{(n-2)\tilde{\mathfrak{W}}^i_{-1,\delta,\eps,\rho}(\bm\alpha;\beta_l)}{(s_l-t)^{n-1}}\cdot
            \end{split}
        \end{equation}
        In order to achieve this, we use that the correlation function $\ps{\Wb_{-n}V_\beta(t)\V}_{\delta,\eps}$ is the sum of four terms $I_0+I_1+I_2+I_3$, corresponding respectively to the free-field (i.e. not integrated) part, and $1$, $2$ and $3$-fold integrals. Likewise the terms in the right-hand side of Equation~\eqref{eq:ward_Wn} take the form $\tilde{I}_0+\tilde{I}_1+\tilde{I}_2$. The equality $I_0=\tilde{I}_0$ has already been checked since it amounts to treating the free-field case. Therefore our goal is to show that the difference $I_1+I_2+I_3-(\tilde I_1+\tilde I_2+\tilde I_3)$ is given by the corresponding remainder up to vanishing terms.
        \\

        \emph{One-fold integrals.} And to start with let us first compute the part in the right-hand side of the Ward identity~\eqref{eq:ward_Wn} that involves only one-fold integrals. Using again the definitions of the descendants we find by Gaussian integration by parts
        \begin{align*}  
			\tilde{I}_1 &= \It \left[ \sum_{k=1}^{2N+M}\left( \frac{-(n-2)(qB(\alpha_k,\gamma e_i)+2C(\alpha_k,\alpha_k,\gamma e_i))+2C(\alpha_k,\gamma e_i,\beta)+2C(\alpha_k,\beta,\gamma e_i)}{2(z_k-t)^{n-1}(z_k-x)}  \right.\right.\\
            &+\frac{q(B(\gamma e_i,\alpha_k)-B(\alpha_k,\gamma e_i))+2C(\gamma e_i-\alpha_k,\alpha_k,\gamma e_i)}{2(z_k-t)^{n-2}(z_k-x)^2}\nonumber\\
            &+ \left.\left.\sum_{l\neq k} \frac{C(\alpha_k,\gamma e_i,\alpha_l)+C(\alpha_k,\alpha_l,\gamma e_i)}{(z_k-t)^{n-2}(z_k-z_l)(z_k-x)} \right)\Psi_i(x)\right] \nonumber\\
            &+\sum_{i=1}^2\mu_{B,i} \int_{\H_{\delta,\eps}} \sum_{k=1}^{2N+M} \frac{2C(\alpha_k,\gamma e_i,\gamma e_i)}{(z_k-t)^{n-2}(x-z_k)(\bar{x}-z_k)}\Psi_i(x) dxd\bar{x}.\nonumber
		\end{align*}
        As for the left-hand side of Equation~\eqref{eq:ward_Wn}, we find by the same method:
        \begin{align*}
			I_1 &= \It \left[ \left(\frac{qB(\beta,\gamma e_i)+2C(\beta,\beta,\gamma e_i)-(n-1)(qB(\gamma e_i,\beta)+2C(\beta,\gamma e_i,\gamma e_i))}{2(x-t)^n} \right.\right.\\
            &+\frac{(n-1)(n-2)(-q^2h_2(\gamma e_i)+\frac{q}{2}B(\gamma e_i,\gamma e_i)+\frac13 C(\gamma e_i,\gamma e_i,\gamma e_i))}{2(x-t)^n}\nonumber\\
            &+ \sum_{k=1}^{2N+M} \left( \sum_{p=1}^{n-1} \frac{(p-1)(qB(\gamma e_i,\alpha_k)+2C(\gamma e_i,\gamma e_i,\alpha_k))-2C(\gamma e_i,\alpha_k,\beta)}{2(x-t)^{p}(z_k-t)^{n-p}}\right. \nonumber\\
            &+\frac{(p-1)(qB(\alpha_k,\gamma e_i) + 2 C(\alpha_k,\alpha_k,\gamma e_i))-2C(\alpha_k,\gamma e_i,\beta)}{(z_k-t)^{p}(x-t)^{n-p}} \nonumber\\
            &+\frac13\sum_{l\neq k}\sum_{p_1=1}^{n-2}\sum_{p_2=1}^{p_1} \left( \frac{C(\alpha_k,\gamma e_i,\alpha_l)}{(z_k-t)^{p_2}(x-t)^{p_1-p_2+1}(z_l-t)^{n-p_1-1}}+\frac{C(\gamma e_i,\alpha_k,\alpha_l)}{(x-t)^{p_2}(z_k-t)^{p_1-p_2+1}(z_l-t)^{n-p_1-1}}\right.\nonumber\\ 
            &\left.\left.\left.\left.+\frac{C(\alpha_k,\alpha_l,\gamma e_i)}{(z_k-t)^{p_2}(z_l-t)^{p_1-p_2+1}(x-t)^{n-p_1-1}} \right)\right) \right)\Psi_i(x)\right] \nonumber\\
            &+\sum_{i=1}^2 \mu_{B,i} \int_{\H_{\delta,\eps}}\left( \sum_{p=1}^{n-1}\frac{(p-1)(qB(\gamma e_i,\gamma e_i)-2C(\gamma e_i,\gamma e_i,\gamma e_i))+2C(\gamma e_i,\gamma e_i,\beta)}{2(x-t)^p(\bar{x}-t)^{n-p}}\right.\nonumber\\
            &+\frac13 \sum_{k=1}^{2N+M} \sum_{p_1=1}^{n-2}\sum_{p_2=1}^{p_1} \left( \frac{C(\alpha_k,\gamma e_i,\gamma e_i) }{(z_k-t)^{p_2}(x-t)^{p_1-p_2+1}(\bar{x}-t)^{n-p_1-1}}+\frac{C(\alpha_k,\gamma e_i,\gamma e_i)}{(x-t)^{p_2}(z_k-t)^{p_1-p_2+1}(\bar{x}-t)^{n-p_1-1}}\right. \nonumber\\
            &\left.\left. +\frac{C(\alpha_k,\gamma e_i,\gamma e_i)}{(x-t)^{p_2}(\bar{x}-t)^{p_1-p_2+1}(z_k-t)^{n-p_1-1}}\right)\right)\Psi_i(x)dxd\bar{x}. \nonumber
		\end{align*}
        First of all algebraic computations show that the $(t-x)^{-n}$ coefficient in the above equals 
        \begin{equation*}
            h_2(e_i)((n-2)q+2\omega_{\hat{i}}(\beta))(n-1+\frac{\ps{\beta,\gamma e_i}}{2}) =: J_{-n}^\beta(e_i).
        \end{equation*}
        Besides, we can  symmetrize the $I_1$ term like before to obtain 
        \begin{align*}
			I_1 &= \It \left[\left(\frac{J_{-n}^\beta(e_i)}{(x-t)^n}+ \sum_{k=1}^{2N+M}\left(\frac{q(B(\gamma e_i,\alpha_k)-B(\alpha_k,\gamma e_i))+2C(\gamma e_i-\alpha_k,\alpha_k,\gamma e_i)}{2(z_k-t)^{n-2}(z_k-x)^2}\right.\right.\right.\\
            &+\frac{q(B(\alpha_k,\gamma e_i)-B(\gamma e_i,\alpha_k))+2C(\alpha_k-\gamma e_i,\alpha_k,\gamma e_i)}{2(x-t)^{n-2}(z_k-x)^2}\\
            &-\frac{-(n-2)(qB(\gamma e_i,\alpha_k)+2C(\gamma e_i,\gamma e_i,\alpha_k))+2C(\alpha_k,\gamma e_i,\beta)+2C(\alpha_k,\beta,\gamma e_i)}{2(x-t)^{n-1}(z_k-x)} \nonumber\\
            &+ \frac{-(n-2)(q B(\alpha_k,\gamma e_i)+2C(\alpha_k,\alpha_k,\gamma e_i))+2C(\alpha_k,\gamma e_i,\beta)+2C(\alpha_k,\beta,\gamma e_i)}{2(z_k-t)^{n-1}(z_k-x)}  \nonumber\\
            &+ \left.\left.\sum_{l\neq k} \frac{C(\alpha_k,\gamma e_i,\alpha_l)+C(\alpha_k,\alpha_l,\gamma e_i)}{(z_k-t)^{n-2}(z_k-z_l)(z_k-x)}+\frac{C(\gamma e_i,\alpha_k,\alpha_l)}{(x-t)^{n-2}(z_k-x)(z_l-x)} \right) \Psi_i(x)\right] \nonumber\\
            &+\sum_{i=1}^2\mu_{B,i} \int_{\H_{\delta,\eps}} \left(\frac{(n-2)(qB(\gamma e_i,\gamma e_i)+2C(\gamma e_i,\gamma e_i,\gamma e_i))-4C(\gamma e_i,\gamma e_i,\beta)}{2(x-t)^{n-1}(x-\bar{x})} \right.\nonumber\\
            &\left.+\sum_{k=1}^{2N+M} \frac{2C(\alpha_k,\gamma e_i,\gamma e_i)}{(x-t)^{n-2}(z_k-x)(x-\bar{x})}+\frac{2C(\alpha_k,\gamma e_i,\gamma e_i)}{(z_k-t)^{n-2}(z_k-x)(z_k-\bar{x})}\right)\Psi_i(x) dxd\bar{x}.\nonumber
		\end{align*}
        Therefore we can write that the difference $I_1-\tilde I_1$ is given by
        \begin{align*}
			&\It \left[\left(\frac{J_{-n}^\beta(e_i)}{(x-t)^n}+ \sum_{k=1}^{2N+M}\left(\frac{q(B(\alpha_k,\gamma e_i)-B(\gamma e_i,\alpha_k))+2C(\alpha_k-\gamma e_i,\alpha_k,\gamma e_i)}{2(x-t)^{n-2}(z_k-x)^2}\right.\right.\right.\\
            &+\frac{(n-2)(qB(\gamma e_i,\alpha_k)+2C(\gamma e_i,\gamma e_i,\alpha_k))-2C(\alpha_k,\gamma e_i,\beta)-2C(\alpha_k,\beta,\gamma e_i)}{2(x-t)^{n-1}(z_k-x)} \nonumber\\
            &+ \left.\left.\sum_{l\neq k} \frac{C(\gamma e_i,\alpha_k,\alpha_l)}{(x-t)^{n-2}(z_k-x)(z_l-x)} \right) \Psi_i(x)\right] \nonumber\\
            &+\sum_{i=1}^2\mu_{B,i} \int_{\H_{\delta,\eps}} \left(\frac{(n-2)(qB(\gamma e_i,\gamma e_i)+2C(\gamma e_i,\gamma e_i,\gamma e_i))-4C(\gamma e_i,\gamma e_i,\beta)}{2(x-t)^{n-1}(x-\bar{x})} \right.\nonumber\\
            &\left.+\sum_{k=1}^{2N+M} \frac{2C(\alpha_k,\gamma e_i,\gamma e_i)}{(x-t)^{n-2}(z_k-x)(x-\bar{x})}\right)\Psi_i(x) dxd\bar{x}.\nonumber
		\end{align*}
        Now we can use the following identities (proved using the explicit expressions for $B$ and $C$):
        \begin{align} \label{eq: identities I1}
            &q B(\gamma e_i,\alpha_k)+2C(\gamma e_i,\gamma e_i,\alpha_k)=q h_2(e_i)\ps{\gamma e_i,\alpha_k}+4h_2(e_i)\omega_{\hat i}(\alpha_k);\\
            &C(\alpha_k,\gamma e_i,\beta)+C(\alpha_k,\beta,\gamma e_i)=-h_2(e_i)\left(\omega_{\hat i}(\beta)\ps{\gamma e_i,\alpha_k}+\omega_{\hat i}(\alpha_k)\ps{\gamma e_i,\beta}\right);\nonumber\\
            &q(B(\alpha_k,\gamma e_i)-B(\gamma e_i,\alpha_k))+2C(\alpha_k-\gamma e_i,\alpha_k,\gamma e_i)=-4h_2(e_i)\omega_{\hat i}(\alpha_k)\left(1+\frac{\ps{\gamma e_i,\alpha_k}}{2}\right) \nonumber
        \end{align}
        to rewrite the above under the form
        \begin{align*} 
			&\It \left[\left(\frac{h_2(e_i)((n-2)q+2\omega_{\hat{i}}(\beta))(n-1+\frac{\ps{\beta,\gamma e_i}}{2})}{(x-t)^n}+ \sum_{k=1}^{2N+M}\left(\frac{-4h_2(e_i)\omega_{\hat i}(\alpha_k)\left(1+\frac{\ps{\gamma e_i,\alpha_k}}{2}\right)}{2(x-t)^{n-2}(z_k-x)^2}\right.\right.\right.\\
            &+\frac{(n-2)(q h_2(e_i)\ps{\gamma e_i,\alpha_k}+4h_2(e_i)\omega_{\hat i}(\alpha_k))+2h_2(e_i)\left(\omega_{\hat i}(\beta)\ps{\gamma e_i,\alpha_k}+\omega_{\hat i}(\alpha_k)\ps{\gamma e_i,\beta}\right)}{2(x-t)^{n-1}(z_k-x)} \nonumber\\
            &+ \left.\left.\sum_{l\neq k} \frac{-4h_2(e_i)\omega_{\hat i}(\alpha_k)\ps{\gamma e_i,\alpha_l}}{(x-t)^{n-2}(z_k-x)(z_l-x)} \right) \Psi_i(x)\right] \nonumber\\
            &+\sum_{i=1}^2\mu_{B,i} \int_{\H_{\delta,\eps}} \left(\frac{(n-2)(qB(\gamma e_i,\gamma e_i)+2C(\gamma e_i,\gamma e_i,\gamma e_i))-4C(\gamma e_i,\gamma e_i,\beta)}{2(x-t)^{n-1}(x-\bar{x})} \right.\nonumber\\
            &\left.+\sum_{k=1}^{2N+M} \frac{2C(\alpha_k,\gamma e_i,\gamma e_i)}{(x-t)^{n-2}(z_k-x)(x-\bar{x})}\right)\Psi_i(x) dxd\bar{x}.\nonumber
		\end{align*}
        Like before we recognize derivatives of the correlation functions. More precisely we end up with
        \begin{align}  \label{eq: one fold int derivatives}
			&I_1-\tilde I_1=-\It \left[\partial_x\left(\left(\frac{-h_2(e_i)((n-2)q+2\omega_{\hat{i}}(\beta))}{(x-t)^{n-1}}+ \sum_{k=1}^{2N+M}\frac{2h_2(e_i)\omega_{\hat i}(\alpha_k)}{(x-t)^{n-2}(z_k-x)}\right)\Psi_i(x)\right)\right]\\
            &+\It\times\It \left[\left(\frac{-h_2(e_i)((n-2)q+2\omega_{\hat{i}}(\beta))\ps{\gamma e_i,\gamma e_j}}{(x-t)^{n-1}(x-y)}+ \sum_{k=1}^{2N+M}\frac{2h_2(e_i)\omega_{\hat i}(\alpha_k)\ps{\gamma e_i,\gamma e_j}}{(x-t)^{n-2}(z_k-x)(x-y)}\right)\Psi_i(x)\right] \nonumber\\
            &+\sum_{i=1}^2\mu_{B,i} \int_{\H_{\delta,\eps}} \left(\frac{(n-2)(qB(\gamma e_i,\gamma e_i)+2C(\gamma e_i,\gamma e_i,\gamma e_i))-4C(\gamma e_i,\gamma e_i,\beta)}{2(x-t)^{n-1}(x-\bar{x})} \right.\nonumber\\
            &\left.+\sum_{k=1}^{2N+M} \frac{2C(\alpha_k,\gamma e_i,\gamma e_i)}{(x-t)^{n-2}(z_k-x)(x-\bar{x})}\right)\Psi_i(x) dxd\bar{x}.\nonumber
		\end{align}
        The first integral corresponds exactly to the one in Lemma~\ref{lemma:desc_wn}. The terms involving $x,\bar{x}$ are treated exactly the same way using the identities \eqref{eq: identities I1} with either $\beta$ or $\alpha_k$ replaced by $\gamma e_i$. The second line of the above will show up when considering the two-fold integrals, which is our next task. \\

        \emph{Two-fold integrals.} The two-fold integrals in the l.h.s. of the Ward identity~\eqref{eq:ward_Wn} are given by (we voluntarily omit the terms involving combinations of the variables $x,y$ with $\bar{x}$ or $\bar{y}$, since they can be dealt with in the exact same way)
        \begin{align*}
			I_2 &= -I_{tot}^2 \left[ \left(\sum_{p=1}^{n-1} \frac{q(p-1)B(\gamma e_i,\gamma e_j) - 2C(\gamma e_i,\gamma e_j,\beta) + 2(p-1)C(\gamma e_i,\gamma e_i,\gamma e_j)}{2(x-t)^{p}(y-t)^{n-p}} \right.\right.\\
            &+\frac13\sum_{k=1}^{2N+M}\sum_{p_1=1}^{n-2}\sum_{p_2=1}^{p_1} \left( \frac{C(\alpha_k,\gamma e_i,\gamma e_j)}{(z_k-t)^{p_2}(x-t)^{p_1-p_2+1}(y-t)^{n-p_1-1}}+\frac{C(\gamma e_i,\alpha_k,\gamma e_j)}{(x-t)^{p_2}(z_k-t)^{p_1-p_2+1}(y-t)^{n-p_1-1}}\right.\nonumber\\ 
            &\left.\left.\left.+\frac{C(\gamma e_i,\gamma e_j,\alpha_k)}{(x-t)^{p_2}(y-t)^{p_1-p_2+1}(z_k-t)^{n-p_1-1}} \right) \right)\Psi_{i,j}(x,y)\right]. \nonumber
		\end{align*}
        Once we have symmetrized all the terms we are left with
        \begin{align*} 
			&-I_{tot}^2 \left[ \left(\left( \frac{1}{(y-x)^2}\left(\frac{1}{(y-t)^{n-2}}-\frac{1}{(x-t)^{n-2}}\right)-\frac{n-2}{(x-y)(x-t)^{n-1}}\right)(\frac{q}{2}B(\gamma e_i,\gamma e_j) + C(\gamma e_i,\gamma e_i,\gamma e_j)) \right.\right.\\
            &-\frac{1}{y-x}\left(\frac{1}{(y-t)^{n-1}}-\frac{1}{(x-t)^{n-1}} \right)C(\gamma e_i,\gamma e_j,\beta) +\sum_{k=1}^{2N+M}\left(\frac{C(\alpha_k,\gamma e_i,\gamma e_j)}{(z_k-x)(z_k-y)(z_k-t)^{n-2}}\right.\nonumber\\
            &\left.\left.\left. + \frac{C(\alpha_k,\gamma e_i,\gamma e_j)+C(\alpha_k,\gamma e_j,\gamma e_i)}{(x-y)(x-z_k)(x-t)^{n-2}}\right) \right)\Psi_{i,j}(x,y)\right]\nonumber.
		\end{align*}
        Using symmetry between the integration variables allows one to reduce the latter to
        \begin{align*}
			 &-\It^2 \left[ \left(-\frac{(n-2)(\frac{q}{2}B(\gamma e_i,\gamma e_j) + C(\gamma e_i,\gamma e_i,\gamma e_j))-C(\gamma e_i,\gamma e_j,\beta)-C(\gamma e_j,\gamma e_i,\beta)}{(x-y)(x-t)^{n-1}} \right.\right.\\
            &-\frac{\frac{q}{2}(B(\gamma e_i,\gamma e_j)-B(\gamma e_j,\gamma e_i)) + C(\gamma e_i,\gamma e_i,\gamma e_j)- C(\gamma e_j,\gamma e_j,\gamma e_i)}{(y-x)^2(x-t)^{n-2}}\nonumber\\
            &\left.\left.\left. + \sum_{k=1}^{2N+M}\frac{C(\alpha_k,\gamma e_i,\gamma e_j)+C(\alpha_k,\gamma e_j,\gamma e_i)}{(x-y)(x-z_k)(x-t)^{n-2}}\right) \right)\Psi_{i,j}(x,y)\right]\nonumber\\
            &-I_{tot}^2\left[\sum_{k=1}^{2N+M} \frac{C(\alpha_k,\gamma e_i,\gamma e_j)}{(z_k-t)^{n-2}(x-z_k)(y-z_k)} \Psi_{i,j}(x,y)\right]. \nonumber
		\end{align*}
        The last line corresponds exactly to the two-fold integral $\tilde{I}_2$ in the r.h.s. of the Ward identity. For the other terms we use the identities (derived from Equation~\eqref{eq: identities I1})
        \begin{equation*}
            \begin{split}
            (n-2)(\frac{q}{2}B(\gamma e_i,\gamma e_j) + C(\gamma e_i,\gamma e_i,\gamma e_j))-C(\gamma e_i,\gamma e_j,\beta)-C(\gamma e_j,\gamma e_i,\beta)\\=2\gamma h_2(e_i)\delta_{i\neq j}\left((n-2)+\frac{\ps{\gamma e_i,\beta}}{2}\right)+h_2(e_i)((n-2)q+2\omega_{\hat{i}}(\beta))\frac{\ps{\gamma e_i,\gamma e_j}}{2};
        \end{split}
        \end{equation*}
        $$
        \frac{q}{2}(B(\gamma e_i,\gamma e_j)-B(\gamma e_j,\gamma e_i)) + C(\gamma e_i,\gamma e_i,\gamma e_j)- C(\gamma e_j,\gamma e_j,\gamma e_i) = 2\gamma h_2(e_i)\left(1+\frac{\ps{\gamma e_i,\gamma e_j}}{2}\right)\delta_{i\neq j};
        $$
        $$
        C(\alpha_k,\gamma e_i,\gamma e_j)+C(\alpha_k,\gamma e_j,\gamma e_i)= -2\gamma h_2(e_i) \frac{\ps{\gamma e_i,\alpha_k}}{2}\delta_{i\neq j}-\ps{\gamma e_i,\gamma e_j} h_2(e_i)\omega_{\hat{i}}(\alpha_k)
        $$
        to get that the difference $I_2-\tilde I_2$ is given by
        \begin{align*} 
			 -I_{tot}^2 &\left[ -2\gamma h_2(e_i)\delta_{i\neq j}\left(\frac{n-2+\frac{\ps{\gamma e_i,\beta}}{2}}{(x-y)(x-t)^{n-1}} +\frac{1+\frac{\ps{\gamma e_i,\gamma e_j}}{2}}{(x-y)^2(x-t)^{n-2}}\right.\right.\\
            & \left.+\sum_{k=1}^{2N+M}\frac{\ps{\gamma e_i,\alpha_k}}{2(x-y)(x-z_k)(x-t)^{n-2}}\right)\Psi_{i,j}(x,y)\\
            &\left.-\left(\frac{h_2(e_i)((n-2)q+2\omega_{\hat{i}}(\beta))\ps{\gamma e_i,\gamma e_j}}{2(x-y)(x-t)^{n-1}} +\sum_{k=1}^{2N+M}\frac{2 h_2(e_i)\omega_{\hat{i}}(\alpha_k)\ps{\gamma e_i,\gamma e_j}}{(x-y)(x-z_k)(x-t)^{n-2}}\right)\Psi_{i,j}(x,y)\right]. \nonumber
		\end{align*}
        The second line corresponds to the two-fold integrals that we obtained from the treatment of one-fold integrals in Equation~\eqref{eq: one fold int derivatives}. As for the first line we rely on the fact that
        \begin{align*}
            &\partial_x \left( \frac{1}{(x-t)^{n-2}(y-x)} \Psi_{i,j}(x,y)\right) =\\
            &\left( \frac{1+\frac{\ps{\gamma e_i,\gamma e_j}}{2}}{(x-t)^{n-2}(y-x)^2} + \frac{\frac{\ps{\gamma e_i,\gamma e_j}}{2}}{(x-t)^{n-2}(y-x)(\bar{y}-x)} + \frac{\frac{\ps{\gamma e_i,\gamma e_i}}{2}}{(x-t)^{n-2}(\bar{x}-x)(y-x)} + \frac{n-2+\frac{\ps{\gamma e_i,\beta}}{2}}{(x-t)^{n-1}(x-y)} \right.\nonumber\\
            &\nonumber\left.+ \sum_{k=1}^{2N+M} \frac{\ps{\gamma e_i,\alpha_k}}{2(x-t)^{n-2}(y-x)(z_k-x)}\right)\Psi_{i,j}(x,y)\\
            &-\It\left[\frac{\ps{\gamma e_i,\gamma e_f}}{2(x-t)^{n-2}(y-x)(z-x)}\Psi_{i,j,f}(x,y,z)\right] .\nonumber
        \end{align*}
        The integral term that shows up is canceled by the three-fold integrals in the left-hand side of~\eqref{eq:ward_Wn}. Namely these are given by
        \begin{align*}
			I_3 &= -I_{tot}^3 \left[ \sum_{p_1=1}^{n-2}\sum_{p_2=1}^{p_1} \frac{C(\gamma e_i,\gamma e_j,\gamma e_f)}{3(x-t)^{p_2}(y-t)^{p_1-p_2+1}(z-t)^{n-p_1-1}}\Psi_{i,j,f}(x,y,z)\right].
		\end{align*}
        Symmetrizing the above and using symmetry between the variables yield
        \begin{align*} 
			I_3 &= -I_{tot}^3 \left[\delta_{i\neq j}\frac{C(\gamma e_i,\gamma e_j,\gamma e_j)}{(x-t)^{n-2}(x-y)(x-z)}\Psi_{i,j,j}(x,y,z)\right].
		\end{align*}
        Again, algebraic computations show that 
        \begin{equation*}
            C(\gamma e_i,\gamma e_j,\gamma e_j) = -\gamma h_2(e_i) \ps{\gamma e_i,\gamma e_j} 
        \end{equation*}
        for $i\neq j$, which allows to conclude that 
        \begin{align} \label{eq: true remainder wn}
			&\ps{\Wb_{-n}V_\beta(t)\prod_{k=1}^NV_{\alpha_k}(z_k)\prod_{l=1}^MV_{\beta_l}(s_l)}_{\delta,\eps}-\\
			&\left(\sum_{k=1}^{2N+M}\frac{-\Wc_{-2}^{(k)}}{(z_k-t)^{n-2}}+\frac{(n-2)\Wc_{-1}^{(k)}}{(z_k-t)^{n-1}}-\frac{(n-1)(n-2)w(\alpha_k)}{2(z_k-t)^n}\right) \ps{V_\beta(t)\prod_{k=1}^NV_{\alpha_k}(z_k)\prod_{l=1}^MV_{\beta_l}(s_l)}_{\delta,\eps}\nonumber\\
            &=\It\left[\partial_x J^i(x)\right]+\It\times \It\left[\partial_x J^{i,j}(x,y)\right].\nonumber
		\end{align}
        where we have introduced the shorthands
        \begin{equation*}
            J^i(x) = \left(\frac{h_2(e_i)((n-2)q+2\omega_{\hat{i}}(\beta))}{(x-t)^{n-1}}+\sum_{k=1}^{2N+M}\frac{2h_2(e_i)\omega_{\hat{i}}(\alpha_k)}{(x-t)^{n-2}(x-z_k)}\right)\Psi_{i}(x)
        \end{equation*}
        and
        \begin{equation*}
            J^{i,j}(x,y) = \frac{2\gamma h_2(e_i)\delta_{i\neq j}}{(x-t)^{n-2}(y-x)}\Psi_{i,j}(x,y).
        \end{equation*}

        \emph{Concluding the proof.} In order to wrap up the proof of Theorem~\ref{thm:ward_Wn}, it only remains to show that
        \begin{align*}
            \lim\limits_{\delta,\eps\to0}\It\left[\partial_x J^i(x)\right]+\It\times \It\left[\partial_x J^{i,j}(x,y)\right]-\sum_{i=1}^2 \mathfrak{R}^i(\delta,\eps,\rho;\bm\alpha)=0.
        \end{align*}
        For this purpose we integrate by parts the right-hand side of Equation~\eqref{eq: true remainder wn}, that is, we apply Stokes formula separately over the domains $\Heps^1\cup\Reps^1$ and $\Heps^c\cup\Reps^c$. Let us first compute the integrals on $\Reps^c$. As already explained, we do not treat the terms at $\infty$ since they are seen to vanish thanks to Lemma~\ref{lemma:inf_integrability_toda}. Around an $s_l$ insertion, we have
        \begin{equation} \label{eq: ibp around sl}
            \begin{split}
                &\int_{\Reps^c\cap (s_l-r_l,s_l+r_l)}\partial_x J^i(x)\mu_i(dx)+\It\left[\int_{\Reps^c\cap (s_l-r_l,s_l+r_l)}\partial_x J^{i,j}(x,y)\mu_i(dx)\right]\\
                &=h_2(e_i)((n-2)q+2\omega_{\hat{i}}(\beta))\left(\mu_i(s_l^-)\frac{\Psi_i(s_l-\eps)}{(s_l-\eps-t)^{n-1}}-\mu_i(s_l^+)\frac{\Psi_i(s_l+\eps)}{(s_l+\eps-t)^{n-1}}\right)\\
                &+\sum_{k\neq l} 2h_2(e_i)\omega_{\hat{i}}(\alpha_k)\left(\mu_i(s_l^-)\frac{\Psi_i(s_l-\eps)}{(s_l-\eps-t)^{n-2}(s_l-\eps-z_k)}-\mu_i(s_l^+)\frac{\Psi_i(s_l+\eps)}{(s_l+\eps-t)^{n-2}(s_l+\eps-z_k)}\right)\\
                &+2h_2(e_i)\omega_{\hat{i}}(\beta_l)\left(\mu_i(s_l^-)\frac{\Psi_i(s_l-\eps)}{(s_l-\eps-t)^{n-2}(-\eps)}-\mu_i(s_l^+)\frac{\Psi_i(s_l+\eps)}{(s_l+\eps-t)^{n-2}\eps}\right)\\
                &+\It\left[2\gamma h_2(e_i)\delta_{i\neq j}\left(\mu_i(s_l^-)\frac{\Psi_i(s_l-\eps)}{(s_l-\eps-t)^{n-2}(y-s_l+\eps)}-\mu_i(s_l^+)\frac{\Psi_i(s_l+\eps)}{(s_l+\eps-t)^{n-2}(y-s_l-\eps)}\right)\right]\\
                &- \text{same terms with }r_l\text{ instead of }\eps.
            \end{split}
        \end{equation}
        The terms in the last line of the above expression cancel out when summing over $l$. Now, we recall that
        \begin{equation*}
            \tilde{\mathfrak{W}}_{-1,\delta,\eps,\rho}^i(\bm\alpha;\beta_l)= -h_2(e_i)\left(q-2\omega_{\hat i}(\beta_l)\right)\left(\mu_{i}(s_{l}^-)\Psi_i(s_l-\eps)-\mu_{i}(s_l^+)\Psi_i(s_l+\eps)\right)\qt{and}
        \end{equation*}
        \begin{align*}
            &\tilde{\mathfrak{W}}_{-2,\delta,\eps,\rho}^i(\bm\alpha;\beta_l)\\
            &= \left(-\sum_{k\neq l}\frac{2h_2(e_i)\omega_{\hat{i}}(\alpha_k)}{z_k-s_l}-\frac{2h_2(e_i)\omega_{\hat{i}}(\beta)}{t-s_l}\right)\left(\mu_{i}(s_l^-)\Psi_i(s_l-\eps)-\mu_{i}(s_l^+)\Psi_i(s_l+\eps)\right) \nonumber\\
            &-\frac{2h_2(e_i)\omega_{\hat{i}}(\beta_l)}{\eps}\left(\mu_{i}(s_l^-)\Psi_i(s_l-\eps)+\mu_{i}(s_l^+)\Psi_i(s_l+\eps)\right)\nonumber\\
            &-\mu_{B,i} \int_{(s_l-r_l,s_l+r_l)} \Im\left(\frac{2\gamma h_2(e_i)\omega_{\hat{i}}(\beta_l)}{x+i\delta -s_l}\right)\Psi_i(x+i\delta)dx\nonumber\\
            &+ \It\left[2\gamma h_2(e_i)\delta_{i\neq j}\left(\mu_{i}(s_l^-)\frac{\Psi_{i,j}(s_l-\eps,y)}{y-s_l+\eps}-\mu_{i}(s_l^+)\frac{\Psi_{i,j}(s_l+\eps,y)}{y-s_l-\eps}\right.\right.\nonumber\\
            &\left.\left.+\mu_{B,i} \int_{(s_l-r_l,s_l+r_l)} \left(\frac{1}{y-x-i\delta}-\frac{1}{y-x+i\delta}\right)\Psi_{i,j}(x+i\delta,y)\frac{idx}{2}\right)\right] \nonumber.
        \end{align*}
        Using these explicit expressions one readily sees that, up to a $o(\eps)$ term, Equation~\eqref{eq: ibp around sl} corresponds to the part that is on the real line (that is that does not contain a term of the form $x+i\delta$) of the expected remainder $$-\sum_{l=1}^M \frac{-\tilde{\mathfrak{W}}^i_{-2,\delta,\eps,\rho}(\bm\alpha;\beta_l)}{(s_l-t)^{n-2}}+\frac{(n-2)\tilde{\mathfrak{W}}^i_{-1,\delta,\eps,\rho}(\bm\alpha;\beta_l)}{(s_l-t)^{n-1}}\cdot$$
        This is true except for the term 
        $$
        \frac{2(n-2)h_2(e_i)\omega_{\hat i}(\beta_l)}{(s_l-t)^{n-1}}\left(\mu_{i}(s_{l}^-)\Psi_i(s_l-\eps)-\mu_{i}(s_l^+)\Psi_i(s_l+\eps)\right)
        $$
        coming from $-(n-2)\frac{\tilde{\mathfrak{W}}_{-1,\delta,\eps,\rho}^i(\bm\alpha;\beta_l)}{(s_l-t)^{n-1}}$. Nonetheless, one can recover this term by performing an expansion of the fourth line in Equation~\eqref{eq: ibp around sl}. Indeed, if one writes
        \begin{equation*}
            \frac{1}{(s_l+\eps-t)^{n-2}\eps} = \frac{1}{(s_l-t)^{n-2}\eps} - \frac{n-2}{(s_l-t)^{n-1}} + o(\eps)
        \end{equation*}
        and
        \begin{equation*}
            \frac{1}{(s_l-\eps-t)^{n-2}\eps} = \frac{1}{(s_l-t)^{n-2}\eps} + \frac{n-2}{(s_l-t)^{n-1}} + o(\eps).
        \end{equation*}
        Then replacing the fourth line in Equation~\eqref{eq: ibp around sl} using these identities yields the missing term.
        
        Let us now turn to the integrals over $\Heps^c$. When integrating by parts the right-hand side of \eqref{eq: true remainder wn} over $\Heps^c$ one obtains two kind of terms, see Figure~\ref{fig:integration around sl} (again, we omit the terms at infinity since they are easily shown to vanish thanks to Lemma~\ref{lemma:inf_integrability_toda}). The first ones are the contour integrals around the $z_k$'s (in blue on Figure~\ref{fig:integration around sl}), that are of the form
        \begin{equation*}
            \begin{split}
            \oint_{\partial B(z_k,\eps)} \Psi_i(\xi)(\tilde{J}^i(\xi)\frac{id\bar{\xi}}{2}-\tilde{J}^i(\bar{\xi})\frac{id\xi}{2})+\It\left[\oint_{\partial B(z_k,\eps)} \Psi_{i,j}(\xi,y)(\tilde{J}^{i,j}(\xi,y)\frac{id\bar{\xi}}{2}-\tilde{J}^{i,j}(\bar{\xi},y)\frac{id\xi}{2})\right].
            \end{split}
        \end{equation*}
        These terms actually tend to $0$ in the $\delta,\eps\to 0$ thanks to the fusion asymptotics. The other term is the integral over $(\R+i\delta \setminus (t-r+i\delta,t+r+i\delta)) \cup (\Heps^c \cap \partial B(t,r))$ (see Figure~\ref{fig:integration around sl}). To deal with it, one has to show that the parts of the integral on the segments $(s_l+r_l+i\delta,s_{l+1}-r_{l+1}+i\delta)$ vanishes in the limit, since the other parts are the singular terms 
        \begin{align*}
            &-\int_{(s_l-r_l,s_l+r_l)} \Im (\tilde{J}^i(x+i\delta)) \Psi_i(x+i\delta) dx \\
            &+ \It \left[\int_{(s_l-r_l,s_l+r_l)} (\tilde{J}^{i,j}(x+i\delta,y)-\tilde{J}^{i,j}(x-i\delta,y))\Psi_{i,j}(x+i\delta,y) \frac{idx}{2} \right]
        \end{align*}
        stemming from the definition of the descendant $\Wb_{-2}^{\beta_l}$ (shown in red on Figure~\ref{fig:integration around sl}). This is true like before thanks to fusion asymptotics. The term in green on Figure~\ref{fig:integration around sl} has already been treated above as it was part of the remainder terms.
        
        \begin{figure}[h!]
            \centering
            \begin{tikzpicture}
                \filldraw [fill=gray, fill opacity=0.1, draw=none] (-5,0.5) -- (5,0.5) -- (5,4) -- (-5,4) -- (-5,0.5);
                \filldraw [fill=white, fill opacity=1, draw=none] (-4.67705,0.5) -- (-1.32295,0.5) arc (16.6015:180-16.6015:1.75);
                \filldraw [fill=white, fill opacity=1, draw=none] (0,2.5) circle (0.5);
                \draw [dotted] (-6,0) -- (-5,0) (-0.5,0) -- (0.5,0) (5,0) -- (6,0) node [anchor=west] {$\mathbb{R}_\varepsilon$} (-6,0.5) -- (-5,0.5) (5,0.5) -- (6,0.5) node [anchor=west] {$\mathbb{R}+i\delta$};
                \draw (-5,0) -- (-3.5,0) (-2.5,0) -- (-0.5,0) (0.5,0) -- (2.5,0) (3.5,0) -- (5,0) (-5,0.5) -- (-4.67705,0.5) (-1.32295,0.5) -- (1.32295,0.5) (4.67705,0.5) -- (5,0.5);
                \draw [dashed] (-3.5,0) -- (-2.5,0) (2.5,0) -- (3.5,0);
                \foreach \Point/\PointLabel in {(-3,0)/t, (3,0)/s_l, (0,2.5)/z_k, (-4.75,0)/t-r, (-1.25,0)/t+r, (1.25,0)/s_l-r_l, (4.75,0)/s_l+r_l}
                \draw [fill=black] \Point circle (0.05) node [anchor=north] {$\PointLabel$};
                \draw [draw=blue](0,2.5) circle (0.5);
                \draw [draw=green, fill=green] (2.5,0) circle (0.05) (3.5,0) circle (0.05);
                \draw [<->] (-3,-0.5) -- (-2.5,-0.5) node [anchor=north east] {$\varepsilon$};
                \draw [<->] (3,-0.5) -- (2.5,-0.5) node [anchor=north west] {$\varepsilon$};
                \draw [<->] (0,3.25) -- (0.5,3.25) node [anchor=south east] {$\varepsilon$};
                \draw (-1.32295,0.5) arc (16.6015:180-16.6015:1.75);
                \draw [->,dashed] (4.67705,0.5) arc (16.6015:180-16.6015:1.75);
                \draw [draw=red] (1.32295,0.5) -- (4.67705,0.5);
                \draw [dashed] (-4.67705,0.5) arc (180-16.6015:180:1.75);
                \draw [dashed] (1.32295,0.5) arc (180-16.6015:180:1.75);
                \draw [dashed] (-1.25,0) arc (0:16.6015:1.75);
                \draw [dashed] (4.75,0) arc (0:16.6015:1.75);
            \end{tikzpicture}
            \caption{The contours of integration in $I_{reg}$}
            \label{fig:integration around sl}
        \end{figure}
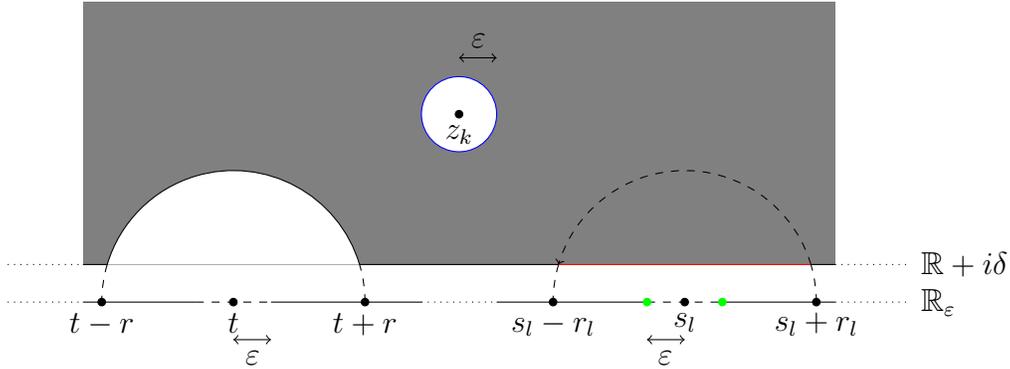
        
        At this point we have proven Lemma~\ref{lemma:desc_wn}. It remains to show that $\tilde{\mathfrak{W}}_{-n,\delta,\eps,\rho}^i(\bm\alpha)$ and $\mathfrak{W}_{-n,\delta,\eps,\rho}^i(\bm\alpha)$ differ by $(P)$-class terms in order to remove the dependence in $r$ and justify the Definition~\ref{defi: desc wn}. For this purpose we integrate by parts the integrals over $\Heps^1\cup\Reps^1$ (the associated domain of integration is depicted on Figure~\ref{fig:integration around t} with the remainder terms shown in red). 
        
        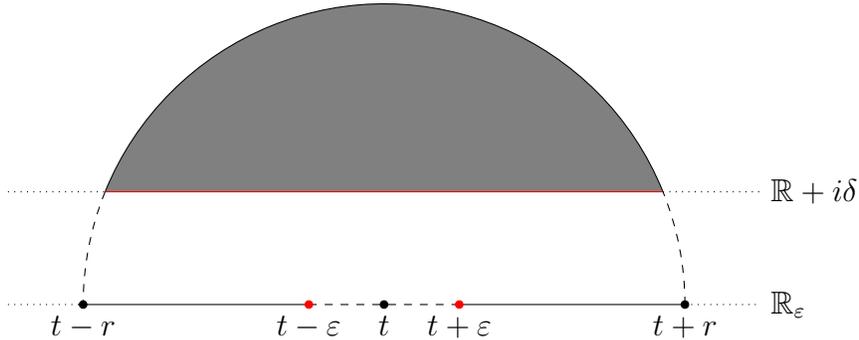
\begin{figure}[h!]
            \centering
            \begin{tikzpicture} 
                \draw (-4,0) -- (-1,0) (1,0) -- (4,0);
                \draw [dotted] (-5,0) -- (-4,0) (4,0) -- (5,0) node[anchor=west] {$\mathbb{R}_\varepsilon$};
                \draw [dashed] (-1,0) -- (1,0);
                \foreach \Point/\PointLabel in {(-4,0)/t-r, (0,0)/t, (4,0)/t+r} \draw [fill=black] \Point circle (0.05) node[anchor=north] {$\PointLabel$};
                \foreach \Point/\PointLabel in {(-1,0)/t-\varepsilon, (1,0)/t+\varepsilon} \draw [fill=red,red] \Point circle (0.05) node[anchor=north, black] {$\PointLabel$};
                \draw [dashed] (4,0) arc (0:22.0243:4) (-3.7081,1.5) arc (180-22.0243:180:4);
                \filldraw [fill=gray, fill opacity=0.1, draw=none] (-3.7081,1.5) -- (3.7081,1.5) arc (22.0243:180-22.0243:4);
                \draw  (3.7081,1.5) arc (22.0243:180-22.0243:4);
                \draw [red] (-3.7081,1.5) -- (3.7081,1.5);
                \draw [dotted] (-5,1.5) -- (-3.7081,1.5) ((3.7081,1.5) -- (5,1.5) node[anchor=west] {$\mathbb{R}+i\delta$};
            \end{tikzpicture}
            \caption{Integration around $t$}
            \label{fig:integration around t}
        \end{figure}
        
        We obtain
        \begin{equation} \label{eq: ibp around t}
            \begin{split}
            &\Is\left[\partial_x J^i(x)\right]+\Is\times \It\left[\partial_x J^{i,j}(x,y)\right]- \tilde{\mathfrak{W}}_{-n,\delta,\eps,\rho}^i(\bm\alpha)\\
            &=\mu_{R,i}J^i(t+r)-\mu_{L,i}J^i(t-r) + \mu_{B,i}\int_{\Heps\cap\partial B(t,r)} \Psi_i(\xi)(\tilde{J}^i(\xi)\frac{id\bar{\xi}}{2}-\tilde{J}^i(\bar{\xi})\frac{id \xi}{2})\\
            &+\It\left[\mu_{R,i}J^{i,j}(t+r,y)-\mu_{L,i}J^{i,j}(t-r,y)\right.\\
            &\left.+ \mu_{B,i}\int_{\Heps\cap\partial B(t,r)} \Psi_{i,j}(\xi,y)(\tilde{J}^{i,j}(\xi,y)\frac{id\bar{\xi}}{2}-\tilde{J}^{i,j}(\bar{\xi},y)\frac{id \xi}{2})\right].
            \end{split}
        \end{equation}
        Here $\tilde{\mathfrak{W}}_{-n,\delta,\eps,\rho}^i(\bm\alpha)$ is defined in Equation~\eqref{eq: remainder wn tilda}, and with $\tilde{J}$ denoting the function $J$ without the correlation function inside. We conclude thanks to the fusion asymptotics that all the terms in the right-hand side of the equality above are $(P)$-class, thus justifying Definition~\ref{defi: desc wn} and completing the proof of Theorem~\ref{thm:ward_Wn}.
	\end{proof}

\subsection{Global Ward identities} \label{subsec:ward_global}
    In the previous subsections we have provided rigorous definitions for the descendant fields, and proved that thanks to them it was possible to infer local Ward identities. As we will now see this will allow us to show that \textit{global} Ward identities are valid for our probabilistic model, putting strong constraints on the correlation functions. Such global Ward identities are obtained by inserting \textit{holomorphic tensors} within correlation functions, which can be seen as the $\L_{-2}V_\beta$ and $\Wb_{-3}V_\beta$ descendants where $\beta=0$. 
    
    To be more specific, let us assume that $\beta=0$ and denote respectively $\SET\coloneqq \L_{-2}V_0$ and $\Wb(t)\coloneqq \Wb_{-3}V_0$. Then we see that 
    \begin{equation*}
        \begin{split}
        &\SET[\Phi]=\ps{Q,\partial^2\Phi}-\ps{\partial\Phi,\partial\Phi}\qt{and}\\
        &\Wb[\Phi]=q^2h_2(\partial^3\Phi)-2qB(\partial^2\Phi,\partial,\Phi)-8h_1(\partial\Phi)h_2(\partial\Phi)h_3(\partial\Phi).
        \end{split}
    \end{equation*}
    They correspond respectively to the \textit{stress-energy tensor} $\SET$ and the \textit{higher-spin current} $\Wb$, already encountered in Subsection~\ref{subsec:SET}. In the language of Vertex Operator Algebra these currents can be thought of as formal power series whose modes generate the $W_3$ algebra.
    
    \subsubsection{Local Ward identities for the currents}
    The local Ward identities are obtained by inserting these currents within correlation functions. Then we see that, in agreement with Theorems~\ref{thm:ward_vir} and~\ref{thm:ward_Wn} we have the following:
    \begin{theorem}\label{thm:ward_local}
        Assume that the weights satisfy $\bm\alpha\in\mc A_{N,M}$ and take $t\in\R$. Then in the sense of weak derivatives 
		\begin{equation*}
			\begin{split}
				&\ps{\SET(t)\prod_{k=1}^NV_{\alpha_k}(z_k)\prod_{l=1}^MV_{\beta_l}(s_l)}=\left(\sum_{k=1}^{2N+M}\frac{\partial_{z_k}}{t-z_k}+\frac{\Delta_{\alpha_k}}{(t-z_k)^2}\right) \ps{\prod_{k=1}^NV_{\alpha_k}(z_k)\prod_{l=1}^MV_{\beta_l}(s_l)}.
			\end{split}
		\end{equation*}
        As for the higher-spin current we have:
        \begin{equation} \label{eq:ward_W}
			\begin{split}
				&\ps{\Wb(t)\prod_{k=1}^NV_{\alpha_k}(z_k)\prod_{l=1}^MV_{\beta_l}(s_l)}\\
				=&\left(\sum_{k=1}^{2N+M}\frac{\Wb_{-2}^{(k)}}{t-z_k}+\frac{\Wb_{-1}^{(k)}}{(t-z_k)^2}+\frac{w(\alpha_k)}{(t-z_k)^3}\right) \ps{\prod_{k=1}^NV_{\alpha_k}(z_k)\prod_{l=1}^MV_{\beta_l}(s_l)}
			\end{split}
		\end{equation}
    \end{theorem}
    In the above the quantities that appear are defined by taking $\beta=0$ in the definition of the $\L_{-2}V_\beta$ and $\Wb_{-3}V_\beta$ descendants.

    \subsubsection{Global Ward identities}
    Based on the local Ward identities we will be able to derive global Ward identities. These are related to the symmetries enjoyed by the model in that they describe some conserved quantities, which translate as strong constraints put on the correlation functions.
    In order to show that such identities hold true we will rely on the conformal covariance of the correlation functions (Equation \eqref{eq:conf_cov}) together with the fact that the stress-energy tensor and the higher-spin current behave like covariant tensors. More specifically we first show that:
    \begin{proposition}\label{prop:cov_tensor}
        Let $\psi\in PSL(2,\R)$ be a M\"obius transform of the upper-half plane $\H$ and $\bm\alpha\in\mc A_{N,M}$. Then:
        \begin{equation*}
            \begin{split}
            &\ps{\SET(t)\prod_{k=1}^NV_{\alpha_k}(z_k)\prod_{l=1}^MV_{\beta_l}(s_l)}=\\
            &\psi'(t)^2\prod_{k=1}^N\norm{\psi'(z_k)}^{2\Delta_{\alpha_k}}\prod_{l=1}^M\norm{\psi'(s_l)}^{\Delta_{\beta_l}}\ps{\SET\left(\psi(t)\right)\prod_{k=1}^NV_{\alpha_k}(\psi (z_k))\prod_{l=1}^MV_{\beta_l}(\psi(s_l))}.
            \end{split}
        \end{equation*}
        for the stress-energy tensor, while for the higher-spin current:
        \begin{equation*}
            \begin{split}
            &\ps{\Wb(t)\prod_{k=1}^NV_{\alpha_k}(z_k)\prod_{l=1}^MV_{\beta_l}(s_l)}=\\
            &\psi'(t)^3\prod_{k=1}^N\norm{\psi'(z_k)}^{2\Delta_{\alpha_k}}\prod_{l=1}^M\norm{\psi'(s_l)}^{\Delta_{\beta_l}}\ps{\Wb\left(\psi(t)\right)\prod_{k=1}^NV_{\alpha_k}(\psi (z_k))\prod_{l=1}^MV_{\beta_l}(\psi(s_l))}.
            \end{split}
        \end{equation*}
    \end{proposition}
    \begin{proof}
        To start with we note that for such maps, in agreement with~\cite[Proposition 4.1]{Toda_OPEWV}, the currents satisfy
    \begin{equation*}
        \SET\left[\Phi\circ\psi+Q\ln\norm{\psi'}\right]=(\psi')^2\SET\left[\Phi\right]\circ\psi\qt{and}\Wb\left[\Phi\circ\psi+Q\ln\norm{\psi'}\right]=(\psi')^3\Wb\left[\Phi\right]\circ\psi.
    \end{equation*}
    As a consequence at the regularized level we have that 
    \begin{align*}
        &\psi'(t)^{3}\ps{\Wb\left(\psi(t)\right)\prod_{k=1}^NV_{\alpha_k}(\psi (z_k))\prod_{l=1}^MV_{\beta_l}(\psi(s_l))}_{\delta,\eps}\\ &=\ps{\Wb\left[\Phi\circ\psi+Q\ln\norm{\psi'}\right](t)\prod_{k=1}^NV_{\alpha_k}(\psi (z_k))\prod_{l=1}^MV_{\beta_l}(\psi(s_l))}_{\delta,\eps}.
    \end{align*}
    Let us now focus on the remainder terms that show up in the definition of
    $\ps{\Wb\left(\psi(t)\right)\V(\psi(\bm z))}$ (the case of the stress-energy tensor is treated in the exact same way): they are defined from
    \begin{align*}
        &\Is\left[\partial_x\left(\left(\frac{q h_2(e_i)}{(x-\psi(t))^{2}}+\sum_{k=1}^{2N+M}\frac{2h_2(e_i)\omega_{\hat{i}}(\alpha_k)}{(x-\psi(t))(x-\psi(z_k))}\right)\psi_{i}(x)\right)\right]\\
			&+\Is\times \It\left[\partial_x\left(\frac{2\gamma h_2(e_i)\delta_{i\neq j}}{(x-\psi(t))(y-x)}\psi_{i,j}(x,y)\right)\right].
    \end{align*}
    We can make the changes of variables $x\leftrightarrow\psi(x)$ and $y\leftrightarrow\psi(y)$ in the above integrals. We can then combine Equation~\eqref{eq:conf_cov} together with the fact that $\Delta_{\gamma e_i}=1$ and the following property of the Mobi\"us transform:
    \begin{equation*}
        (\psi(x)-\psi(t))(\psi(x)-\psi(y))=\psi'(x)\psi'(t)^{\frac12}\psi'(y)^{\frac12}(x-t)(x-y)
    \end{equation*}
    to get that the remainder term is actually given by $\psi'(t)^{-3}\prod_{k=1}^N\norm{\psi'(z_k)}^{-2\Delta_{\alpha_k}}\prod_{l=1}^M\norm{\psi'(s_l)}^{-\Delta_{\beta_l}}$ times the remainder term that would come from the definition of $\ps{\Wb(t)\V}$. This allows to take the $\delta,\eps\to0$ limit and conclude for the proof.
    \end{proof}
    
    Based on the above property we can readily deduce that global Ward identities hold true:
    \begin{theorem}\label{thm:ward_global}
		Assume that $\bm\alpha\in\mc A_{N,M}$. Then the following identities hold true for $0\leq n\leq 2$ and $0\leq m\leq 4$:
		\begin{equation*}
			\begin{split}
				&\left(\sum_{k=1}^{2N+M}z_k^n\Lc_{-1}^{(k)}+nz_k^{n-1}\Delta_{\alpha_k}\right)\ps{\prod_{k=1}^NV_{\alpha_k}(z_k)\prod_{l=1}^MV_{\beta_l}(s_l)}=0\\
				&\left(\sum_{k=1}^{2N+M}z_k^m\Wc_{-2}^{(k)}+mz_k^{m-1}\Wc_{-1}^{(k)}+\frac{m(m-1)}2 z_k^{m-2}w(\alpha_k)\right)\ps{\prod_{k=1}^NV_{\alpha_k}(z_k)\prod_{l=1}^MV_{\beta_l}(s_l)}=0.
			\end{split}
		\end{equation*}
	\end{theorem}
    \begin{proof}
        This is an immediate consequence of Proposition~\ref{prop:cov_tensor} combined with Theorem~\ref{thm:ward_local}. Namely by taking $\psi(z)=\frac{-1}{z}$ we see that as $t\to\infty$:
			\begin{align*}
				&\left(\sum_{k=1}^{2N+M}\frac{\Wb_{-2}^{(k)}}{t-z_k}+\frac{\Wb_{-1}^{(k)}}{(t-z_k)^2}+\frac{w(\alpha_k)}{(t-z_k)^3}\right) \ps{\prod_{k=1}^NV_{\alpha_k}(z_k)\prod_{l=1}^MV_{\beta_l}(s_l)}\\
                &=t^{-6}\prod_{k=1}^{2N+M}\norm{z_k}^{-2\Delta_{\alpha_k}}\ps{\Wb\left(\frac 1t\right)\prod_{k=1}^{2N+M}V_{\alpha_k}\left(\frac{-1}{z_k}\right)}.
			\end{align*}
			From this we deduce that the left-hand side must scale like $C t^{-6}$ for some constant $C$. But we can expand it in negative powers of $t$: the coefficients that appear in this expansion are precisely given by the ones that appear in the statement of the global Ward identities, namely $\left(\sum_{k=1}^{2N+M}z_k^m\Wc_{-2}^{(k)}+mz_k^{m-1}\Wc_{-1}^{(k)}+\frac{m(m-1)}2 z_k^{m-2}w(\alpha_k)\right)\ps{\prod_{k=1}^NV_{\alpha_k}(z_k)\prod_{l=1}^MV_{\beta_l}(s_l)}$. This is the end, of our elaborate plans, the end.
    \end{proof}